\documentclass[reqno,12pt]{amsart}
\usepackage{amsmath}
\usepackage{amsfonts}
\usepackage{amssymb}
\usepackage{amstext}
\usepackage{amsbsy}
\usepackage{amsopn}
\usepackage{amsthm}
\usepackage{amsxtra}
\usepackage{upref}
\usepackage{fancyhdr,amsmath, graphicx, psfrag, color ,amsfonts,layout, amssymb}
\usepackage{euscript,epsfig}

\newtheorem{thm}{Theorem}[section]

\newtheorem{lem}[thm]{Lemma}
\newtheorem{prop}[thm]{Proposition}

\newtheorem{rem}[thm]{Remark}
\theoremstyle{definition}

\theoremstyle{remark}

\def \R{\mathbb R}
\def \N{{\mathbb N}}

\def \Z{\mathbb Z}

\def \H{\mathbb H}

\def\C{\mathbb C}

\def\Q{\mathbb Q}
\def\D{\mathbb D}
\newcommand{\vv}{{\mathbf{v}}}
\newcommand{\vu}{{\mathbf{u}}}

\newcommand{\supp}{{\rm supp}}
\newcommand\eq[2]{\begin{equation}\label{eq: #1}{#2}\end{equation}}
\newcommand{\SL}{\operatorname{SL}}
\newcommand{\PSL}{\operatorname{PSL}}

\begin{document}

\author{ Fran\c cois MAUCOURANT}
\address{Universit\'e Rennes I, IRMAR UMR 6625, Campus de Beaulieu 35042 Rennes
cedex -  France}
\email{francois.maucourant@univ-rennes1.fr}

\author{Barbara SCHAPIRA}
\address{L.A.M.F.A. UMR 6140
Universit\'e Picardie Jules Verne
33 rue St Leu 80000 Amiens - France}
\email{barbara.schapira@u-picardie.fr}

\title[Distribution of orbits]{Distribution of orbits in $\R^2$ of a finitely 
generated group of $\SL(2,\R)$} 

\begin{abstract}
In this work, we study the asymptotic distribution of the non discrete orbits of a finitely 
generated group acting linearly on $\R^2$. To do this, we establish new
 equidistribution results for the horocyclic flow on the unitary tangent bundle of the associated surface.
\end{abstract}

\maketitle


\section{Introduction}


\subsection{Problem and State of the art}
 Let $\Gamma_0$ be a discrete subgroup of $G_0=\SL(2,\R)$, acting on the plane $\R^2$. 
The subject of understanding the distribution of the orbits of $\Gamma_0$ on $\R^2$ was
 initiated by Ledrappier \cite{led}, who proved that if $\Gamma_0$ is a lattice containing $-I$,
 and if $\Gamma_T$ is the subset
 $$\Gamma_T=\{\gamma \in\Gamma_0 \, : \, ||\gamma||\leq T \},$$
 for the $l^p$-norm on matrices, $p \in [1,+\infty]$, then for any $\vu \in \R^2$ with dense
 $\Gamma_0$-orbit and any continuous test function 
$f$ from $\R^2$ to $\R$,
we have  
 \eq{eq2}
 {
 \lim_{T\rightarrow +\infty} \frac1T \sum_{\gamma \in \Gamma_T} f(\gamma \vu ) = 
\frac{2}{\mu(\Gamma_0\backslash G_o)}
 \int_{\R^2} \frac{f(\vv)}{| \vu |.| \vv |} d \vv,
 }
 where here $|.|$ stands for the usual $l^p$-norm on $\R^2$, and $\mu(\Gamma_0\backslash G_o)$ 
is the covolume of $\Gamma_0$ with its usual normalization. 
   Independently, Nogueira \cite{no} found an alternative proof of 
this theorem in the important case $\Gamma_0=\SL(2,\Z)$, which did not 
involve the study of the horocyclic flow, as in Ledrappier's, 
but purely arithmetic considerations.

 Various generalizations or strenghtenings of this result have been considered: 
on $\C^2$ and Clifford algebras \cite{lp1}, 
on the $p$-adic plane \cite{lp2}, on $\R^n$ for $n \geq 3$ \cite{g1},
 on other homogeneous manifolds \cite{gw}, with remainder terms \cite{no2}, \cite{MW}, \cite{Pollicott}. 

In all of the aforementionned results, one assumes that $\Gamma_0$ is a lattice in the appropriate group.
 However, in \cite{led2}, Ledrappier manages to deal 
with the case when $\Gamma_0$ is a the fundamental group of an abelian cover 
of a compact hyperbolic surface; he showed that in this case, 
(\ref{eq: eq2}) holds for Lebesgue-almost all vectors $\vu$, 
with the normalisation $1/T$ replaced by an appropriate one;
 he also considered the other invariant ergodic locally finite measures 
on $\R^2$ constructed by Babillot-Ledrappier and proved that (\ref{eq: eq2})
 holds almost surely but in the sense of log-Ces\`aro-averages.\\
 
 The purpose of this paper is to deal with the case where $\Gamma_0$ is a nonelementary
 finitely generated discrete subgroup of $G_o$, without torsion elements other than $-Id$. 
Equivalently, the surface $\Gamma_0\backslash \H$, where $\H$ is the hyperbolic plane,
 is geometrically finite. \\


\subsection{Equidistribution of the horocyclic flow} 
Except in \cite{no2}, all results described above rely on strong ergodic properties of the horocyclic 
flow. On the unit tangent bundle of a finite volume hyperbolic surface, all non-closed horocycles are equidistributed
towards the Liouville measure, 
which is the unique ergodic invariant probability measure of 
full support (Furstenberg \cite{Furstenberg}, Dani-Smillie \cite{DS}). 
In the case of abelian covers of compact hyperbolic surfaces, 
Ledrappier's result in \cite{led2} relies on 
the ergodicity of a family of (infinite) invariant ergodic Radon 
measures for the horocyclic flow, and among them
the Liouville measure (see Babillot-Ledrappier \cite{BL} and Sarig \cite{Sarig}). 

We follow the classical strategy. 
On the unit tangent bundle of a geometrically finite hyperbolic surface $S$, two measures are of
particular importance. The measure of maximal entropy of the geodesic flow, also called Bowen-Margulis-
Patterson-Sullivan measure, denoted here by $m^{ps}$, 
is a finite ergodic invariant measure for the geodesic flow, of full support 
in the non wandering set $\Omega$ of the geodesic flow. This measure and the set $\Omega$ are
 not invariant under the horocyclic flow. The nonwandering set $\mathcal{E}$ of the horocyclic flow is 
the union of horocycles intersecting $\Omega$. 
The horocyclic flow has a unique ergodic invariant measure of 
full support on $\mathcal{E}$ (\cite{Burger}, \cite{Roblin})(see \S \ref{2}),
 which is strongly related to
$m^{ps}$. It was recently named 
the Burger-Roblin measure, and we denote it by $m$. 
The critical exponent $\delta$ of the group $\Gamma=\pi_1(S)$ is defined as
 the exponential growth rate of the orbits of $\Gamma$ 
on $\H$. More precisely,
$\delta=\limsup_{T\to +\infty}\frac{1}{T}\log \#\{\gamma\in\Gamma, d(o,\gamma.o)\le T\}$,
 for any fixed point 
$o\in\H$. 

An essential ingredient in our study is the following equidistribution result. 
 
\begin{thm}\label{magic-equivalent}
 Let $S$ be a nonelementary geometrically finite hyperbolic surface. 
There is a nonnegative continuous function $\tau$ on $T^1S$, such that the following holds.
Let $u\in \mathcal{E}$ be a nonwandering and non-periodic vector for 
the horocyclic flow.  

If $f:T^1S\to \R$ is continuous with compact support, then
$$
\lim_{t\rightarrow +\infty} \frac{1}{t^{\delta}\tau(g^{\log t}u)}\int_{-t}^t f(h^su)ds 
=\frac{1}{m^{ps}(T^1S)}\int_{T^1S}f\,dm \,.
$$
Moreover, $t^{\delta } \tau (g^{\log t} u)\rightarrow +\infty$ 
as $t\rightarrow +\infty$. 

If the surface $S$ is convex-cocompact, the nonwandering set $\Omega\subset \mathcal{E}$ of the geodesic flow
 is compact, the map $\tau$ is bounded from below and from above on $\Omega$,  and the above convergence is uniform in $u\in\Omega$.
\end{thm}

Remark that if $f$, $g$ are two continous maps with compact support from $T^1S$ to $\R$, 
we retrieve the ratio equidistribution result of \cite[Th 1.1]{Scha2}: 
$$
\lim_{t\rightarrow +\infty} \frac{\int_{-t}^t f(h^su)ds }{\int_{-t}^t g(h^su)ds} 
=
\frac{\int_{T^1S}f\,dm }{\int_{T^1S}g\,dm}\,.
$$
Thus, theorem \ref{magic-equivalent} above seems apparently stronger than this ratio convergence. 
In fact, this improvement of the equidistribution statement has been obtained here by following the 
arguments of \cite{Scha2}.

Geometrically, the function $\tau$ in the above theorem is important. 
As we will see in the proof of theorem \ref{magic-equivalent}, 
it is the measure of a one-dimensional ball of radius $1$ and center $u$ 
on the horocycle $(h^su)_{s\in\R}$, for the conditional measure of $m^{ps}$. 
More precisely, we have 
$$
\tau(u)=\mu_{H^-(u)}((h^su)_{|s|\le 1})\quad\mbox{and}\quad  t^{\delta}\tau(g^{\log t}u)=
\mu_{H^-(u)}((h^su)_{|s|\le t})\,,
$$
where $\mu_{H^-}(u)$ is the conditional measure of the Patterson-Sullivan 
measure on the strong stable 
horocycle $H^-(u)=(h^su)_{s\in\R}$. 
In particular, $\tau$ is a continuous map, positive only on a  neighbourhood 
at bounded distance of $\Omega$, and zero outside.
When $\Omega$ is a compact set (i.e. $S$ is convex-cocompact), $\tau$ is bounded. 
When $S$ is geometrically finite, with infinite volume and cusps, we will see in  Proposition \ref{mesure-des-boules-horospheriques}
that, up to multiplicative constants, $\tau(u)$ is equivalent to $e^{(1-\delta)d( u, K)}$, for $K$ an arbitrary fixed compact set. 

\begin{rem} \rm  By what precedes, we see that the  quantity $\tau(g^{\log t}u)$ is geometrically very simple to understand, and oscillates
between $1$ when $g^{\log t}u$ belongs to $K$, and $t^{1-\delta}$ when $g^{\log t}u$ is as far as possible from $K$. 
Therefore, the Birkhoff integral oscillates (up to multiplicative constants) between $t^\delta$ and $t$ times $\frac{\int_{T^1M} f dm}{m^{ps}(T^1M)}$.  
Let us emphasize that this kind of statement, with a precise equivalent of a Birkhoff integral, is 
quite rare in infinite ergodic theory. 
\end{rem}

 For a surface $S$ of finite volume, we have
$\delta=1$, and with our normalizations,  $\tau=2/\pi$, and $\pi^2 m^{ps}=\pi m=Liou$, 
where $Liou$ is the usual Liouville measure.
(See also \cite[Prop. 10]{PP} for explicit comparisons between Liouville and other measures, and other notational conventions.)\\


 \subsection{Orbit distribution on the plane}

 Let $\Lambda \subset \mathbb{P}^1$ be the limit set of $\Gamma_0$ 
and $\mathcal{C}(\Gamma_0)\subset \R^2\setminus\{0\}$ be the cone of vectors
 whose projective component lies in $\Lambda$.
\begin{figure}[ht!]
\epsfig{file=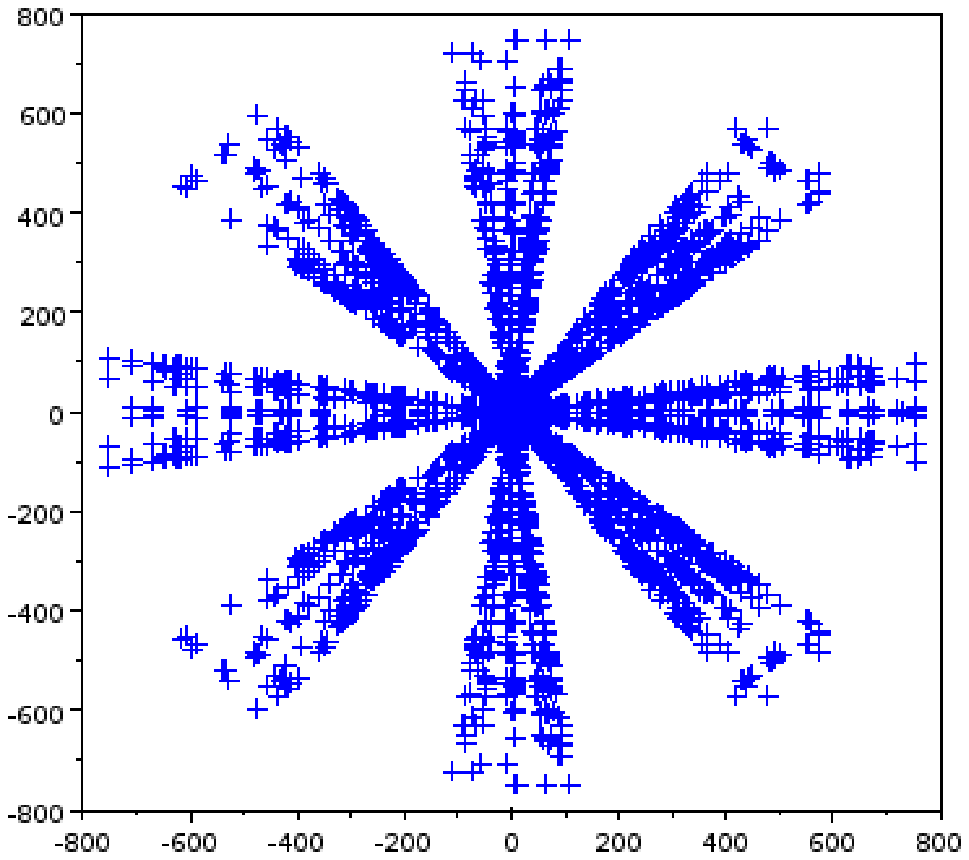,width=150.00pt,height=150.00pt}
\epsfig{file=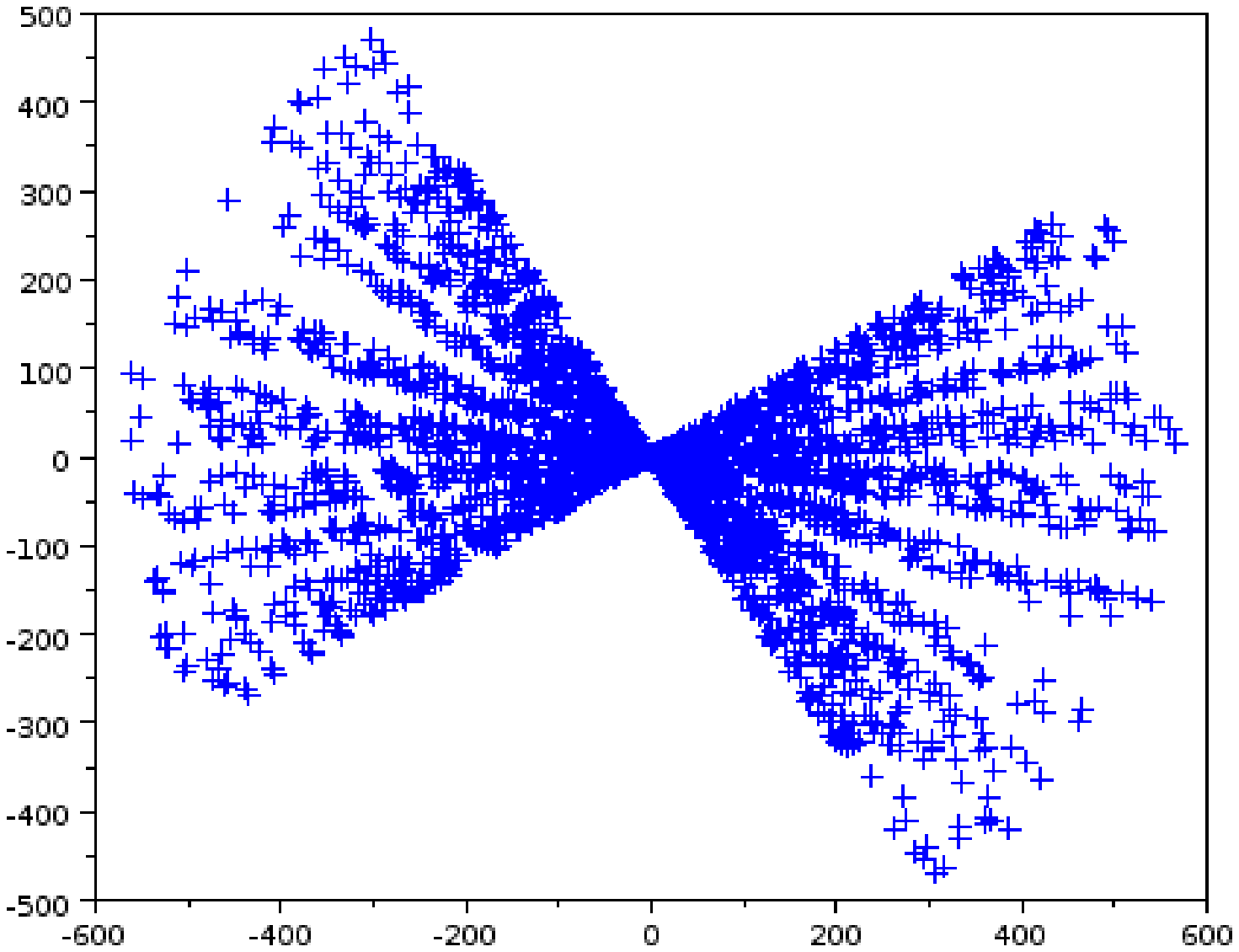,width=150.00pt,height=150.00pt}
\caption{The cloud for a convex-cocompact free group (left) and for a free group with a parabolic element (right)}
\epsfig{file=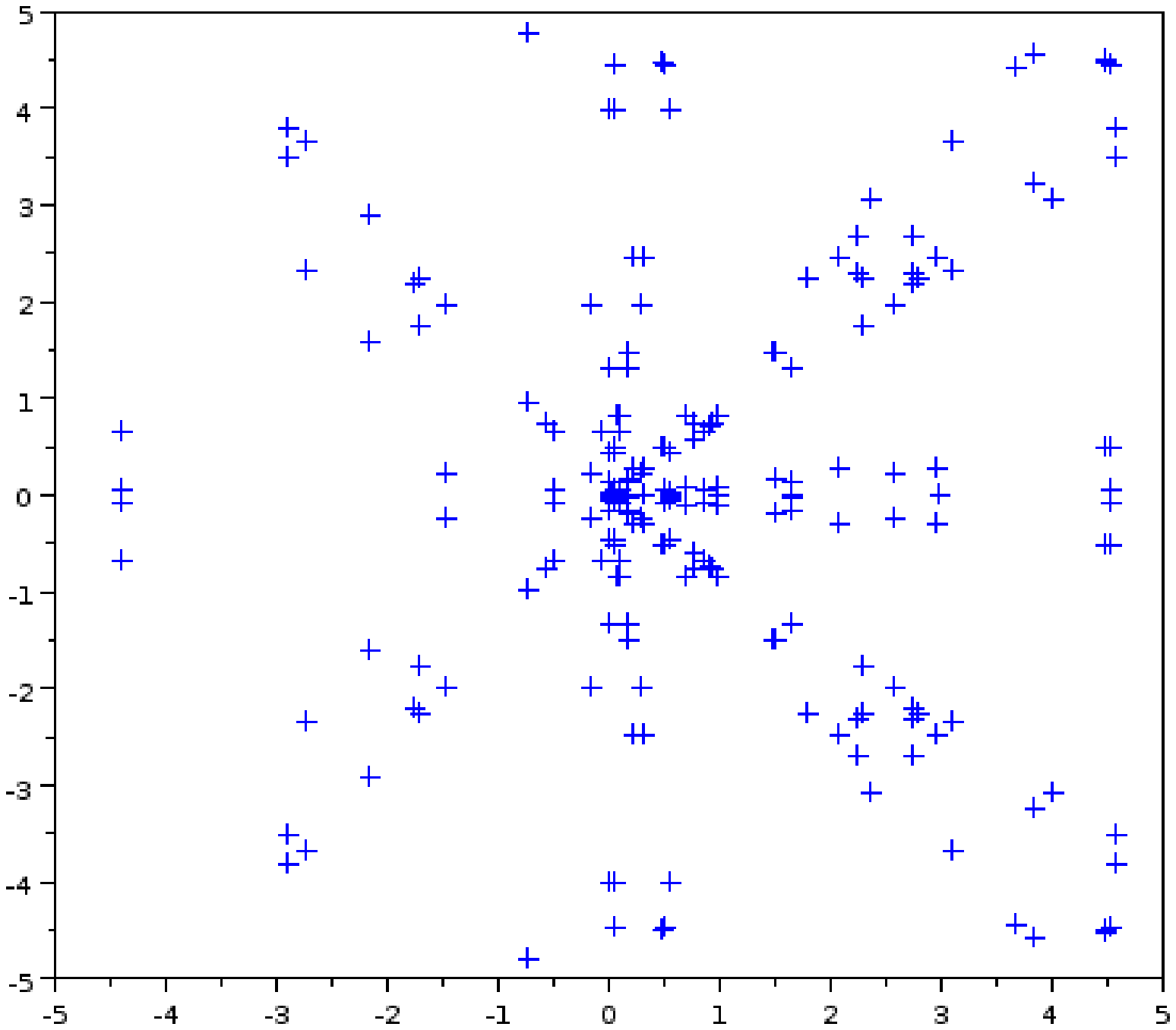,width=200.00pt,height=200.00pt}
\caption{Local distribution for a convex-cocompact free group}
\end{figure} 
 
  This set carries a unique
 (up to scalar multiple) $\Gamma_0$-invariant ergodic measure $\bar{\mu}$ 
of full support, which in polar coordinates is written
 $d\bar{\mu}=2r^{2\delta-1}dr d\bar{\nu}_o$, where $\bar{\nu}_o$
 is the symmetric lift of an appropriate Patterson measure and $\delta$ 
the critical exponent of $\Gamma_0$ (see \S \ref{2} and \S \ref{Relahoro}).

When $S=\Gamma_0\backslash \H$ is of finite volume, with our normalizations,
 $$
\pi \bar{\mu}=\mathcal{L}, \mbox{ where }\mathcal{L} \mbox{ is the Lebesgue measure on }\R^2.$$
  
 In the case of a convex-cocompact group, we show that an analogue of 
(\ref{eq: eq2}) holds 'up to multiplicative constants'.

Let us introduce first a notation. As 
  shown in \cite{gw}, it is possible to  consider an arbitrary norm $||.||$ 
 rather than a $l^p$ norm on $M(2,\R)$ in the definition of $\Gamma_T$. To do that, 
 one should replace the product $|\vu|.|\vv|$ in the right-hand side 
of (\ref{eq: eq2}) by the expression $\vv\star \vu$, which is defined by
$$
\vv\star \vu=
\left\| \left( \begin{array}{cc} - \vu_y\vv_x & \vu_x\vv_x \\ -\vu_y\vv_y & \vu_x\vv_y \end{array} \right) \right\|.
$$
In the case where $\|.\|$ (resp. $|.|$) 
is the $l^p$-norm on $M(2,\R)$ (resp. on $\R^2$), one can check that 
$\vv\star\vu=|\vv||\vu|$. 

Our first result can now be stated.
\begin{thm} \label{th:convexcocompact}
Let $\Gamma_0<\SL(2,\R)$ a convex-cocompact group, which contains $-I$ as unique element of torsion.  
Let $\alpha \in (-1,1)$ be a scaling factor.
For all $\vu \in \mathcal{C}(\Gamma_0)$, for all nonzero, 
nonnegative, continuous and compactly supported functions $f$ on $\R^2\setminus\{0\}$,
we have as $T \rightarrow +\infty$:
\eq{resultat1}{
\frac{1}{T^{(1+\alpha)\delta}}\sum_{\gamma \in \Gamma_T}f\left(\frac{\gamma \vu}{T^\alpha}\right)
\asymp \int_{\R^2} \frac{f(\vv)}{(\vv \star \vu)^{\delta}}d\bar{\mu}(\vv),
}
where the implied constants do not depend on $\vu$ nor on $f$.
\end{thm}

The reader who is not familiar with the subject can consider only the case $\alpha=0$. The scaling factor $T^\alpha$ is interesting for the following reason. 
A typical orbit $\Gamma.\vu$ has no reason to stay away from $0$ and $\infty$. In particular, as $f$ is compactly supported, 
among all elements of $\Gamma_T.\vu$, those belonging to the support of $f$ are very few (of the order of $T^\delta$). The theorem above for $\alpha=0$
is somehow a large deviations result about the rare elements of the orbit $\Gamma.\vu$ in the support of $f$.  
It is therefore interesting to rescale the picture, to see more and more elements of $\Gamma.\vu$, near $0$ (when $\alpha<0$) or $+\infty$ (when $\alpha>0$). 
The largest interesting factor is $\alpha=1$, as it will be seen in theorem \ref{largescale} below. 
 

\begin{rem}\label{boundedorbits}\rm This result is also true for general nonelementary 
geometrically finite groups without torsion except $-I$, 
if we restrict to all $\vu\in \mathcal{C}(\Gamma_0)$,
which correspond on the unit tangent bundle $T^1S=PSL(2,\R)/\Gamma_0$ 
of the corresponding hyperbolic surface to vectors $u$ whose generated geodesic ray is bounded, with a constant
depending on the geodesic ray $(g^tu)_{t\ge 0}$ but not on $f$. 
This will be clear in the proof of theorem \ref{th:convexcocompact}.
However, this set of vectors is of $m^{ps}$-measure zero 
on a geometrically finite surface with cusps. 
\end{rem}

The symbol $a(T) \asymp b(T)$ means that the ratio $a(T)/b(T)$ lies 
between two positive constants for $T$ sufficiently large.
Note that for any vector $\vu \notin \mathcal{C}(\Gamma_0)$, there is a
 constant $c>0$ (depending on $u$) 
such that $|\gamma \vu|\geq c ||\gamma||$,
 so the condition $\vu \in \mathcal{C}(\Gamma_0)$ 
is clearly necessary in the previous Theorem.

Note that the symbol $\asymp$ in (\ref{eq: resultat1}) 
cannot be replaced by a limit
 in a strong sense, namely:

\begin{prop} \label{theresnolimit}
 Assume that $\Gamma_0$ is nonelementary, geometrically finite, with infinite volume, 
and  contains $-I$ as unique element of torsion.  
There exists $f$,$g$ as in Theorem \ref{th:convexcocompact} with
 $\int fd\bar{\mu}>0$, $\int gd\bar{\mu}>0$, such that for $\bar{\mu}$-almost every $\vu$, the ratio
$$
\frac{\sum_{\gamma \in \Gamma_T}f(\gamma \vu)}{\sum_{\gamma \in \Gamma_T}g(\gamma \vu)},
$$
has no limit as $T\rightarrow +\infty$.
\end{prop}


 However, the variations for the ratio between the left hand-side and 
right-hand side of (\ref{eq: resultat1}) disappear under average. 
More precisely, as in the case of $\Z^d$-covers of a 
compact surface \cite{led2}, we obtain an almost-sure log-Ces\`aro 
convergence, under the hypothesis that $\Gamma_0$ is convex-cocompact, 
or geometrically finite with critical exponent $\delta>2/3$.
 This kind of Log-Cesaro-average convergence can be compared to results of Fisher, in  \cite{Fisher}  for 
example. 

\begin{thm}\label{almost-sure-cesaro} Assume that $\Gamma_0$ is
 a nonelementary group containing $-I$ as unique element of torsion. 
Write $S=\Gamma_0\backslash \H$. 
\begin{enumerate}
\item If $\Gamma_0$ is convex-cocompact, then, 
 with the same notations as in Theorem \ref{th:convexcocompact}, 
we have for $\bar{\mu}$-almost every $\vu\in\mathcal{C}(\Gamma_0)$,
\eq{logcesaro}{ 
\lim_{S\rightarrow +\infty} \frac{1}{\log S}\int_1^S \frac{1}{T^{(1+\alpha)\delta}} 
 \sum_{\gamma \in \Gamma_T}f\left(\frac{\gamma \vu}{T^\alpha}\right) \frac{dT}{T}
= \frac{2\int_{T^1S} \tau dm^{ps}}{\left(m^{ps}(T^1S)\right)^2}
\int_{\R^2} \frac{f(\vv)}{(\vv \star \vu)^{\delta}}d\bar{\mu}(\vv).   
}
The function $\tau$ is the same as in Theorem \ref{magic-equivalent}.
\item If $\Gamma_0$ is geometrically finite with cusps,   with  
critical exponent   $\delta>2/3$, and $\alpha=0$, then 
we have for $\bar{\mu}$-almost every $\vu\in\mathcal{C}(\Gamma_0)$,
\eq{logcesaro-gf}{ 
\lim_{S\rightarrow +\infty} \frac{1}{\log S}\int_1^S \frac{1}{T^{ \delta}} 
 \sum_{\gamma \in \Gamma_T}f\left( \gamma \vu \right) \frac{dT}{T}
= \frac{2\int_{T^1S} \tau dm^{ps}}{\left(m^{ps}(T^1S)\right)^2}
\int_{\R^2} \frac{f(\vv)}{(\vv \star \vu)^{\delta}}d\bar{\mu}(\vv).   
} 
\end{enumerate}
\end{thm}


 Of course,  the above formula makes sense only when  $\tau \in L^1(m^{ps})$. 
In the convex-cocompact case, $\tau$ is continuous and $m^{ps}$ has 
compact support, so this is automatic. If the surface $S$ has cusps, we prove: 
 
\begin{thm}\label{integrability-of-balls}
Let $S$ be a geometrically finite surface with cusps. 
The map $\tau$ is integrable w.r.t. $m^{ps}$ if and only if 
the critical exponent of $\Gamma$ satisfies $\delta>2/3$. 
\end{thm}

  This result is surprising. 
Indeed, there are a lot of results proved 
under the assumption $\delta>1/2$. 
It is often for technical reasons (use of methods of harmonic analysis). 
However, at our knowledge, the condition $\delta>2/3$ never 
appeared in the litterature on the subject.

 If $\Gamma_0\backslash \H$ is a geometrically finite surface with cusps, 
then  for $m^{ps}$-almost every $u\in T^1S$, Equation (\ref{eq: resultat1}) 
does not hold anymore: the ratio between 
the left-hand side and the right-hand side is still bounded from below, 
but not from above; and if moreover, the critical exponent satisfies $\delta\leq 2/3$, 
then the same thing happen to Equation (\ref{eq: logcesaro}).
 
\begin{rem}\rm In fact, with exactly the same proof, we give in theorem \ref{Gibbs} a version of 
theorem \ref{almost-sure-cesaro} describing the behaviour of $\bar{\mu}^\varphi$-almost every
$\vu\in\R^2\setminus\{0\}$, where $\bar{\mu}^\varphi$ is the measure on $\R^2\setminus\{0\}$ 
induced by the Gibbs measure $m^\varphi$ on $T^1S$ associated with a H\"older potential $\varphi:T^1S\to \R$. 
\end{rem}


 \subsection{Large scale picture of the cloud $\Gamma_T\vu$}
 We generalize \cite[Cor 1.2]{mau} to this setup, by describing the picture
 which correspond to the case of scaling parameter $\alpha=1$. 

Observe that for large $T$, the set $\{\gamma\in\Gamma, \|\gamma\|\le T,\gamma.\vu\in Supp(f)\}$ 
is very small (comparable to $T^\delta$) compared to $|\Gamma_T|$, which is equivalent to $cT^{2\delta}$. 
Thus, we are interested here in rescaling the orbit $\{\gamma.\vu, \, \|\gamma\| \leq T \}$ in such a way that we can observe all the points.

 In this case, there is no need of assuming that the initial vector  $\vu$ lies in $\mathcal{C}(\Gamma_0)$. 
Indeed, theorem \ref{largescale} relies on an equidistribution result of horocycles pushed by the geodesic
flow (theorem \ref{magic-equivalent-flowed}), 
due to Thomas Roblin, in the spirit of a theorem of Sarnak \cite{Sar}, which is stated in \S \ref{3}. 
And this "flowed equidistribution result" is valid for all vectors $u\in T^1S$.

Introduce first some notations. 
Let $\Psi:\R^2\setminus\{0\} \to PSL(2,\R)$ be the natural section 
(independent of the chosen norm on $M(2,\R)$) 
which associates 
to a vector $\vu=k.a.(1,0)\in\R^2$ the element $ka\in PSL(2,\R)$, 
when we write $\PSL(2,\R)=KAN$ (see \S \ref{defpsi} for details). 
Define the following quantity when $\PSL(2,\R)$ is endowed with the $l^2$-norm:
 $$
\Theta(\vu,\vv)=\frac{1}{|\vu|.|\vv|}\sqrt{1-\frac{|\vv|^2}{|\vu|^2}}.
$$
and $\Theta^m(\vu,\vv)$ is a map here defined to be $0$. In the case of another norm $\|.\|$ on matrices, the expression of $\Theta(\vu,\vv)$ and $\Theta^m(\vu,\vv)$ are different.   
Define $\mathcal{D}_0(\vu)=D(0,|\vu|)$ in the case of the $l^2$-norm, 
and $\mathcal{D}_0(\vu)=\{\vv\in\R^2,\Theta(\vu,\vv)>0\}\cup\{0\}$ for other norms.  

  \begin{thm} \label{largescale} Endow $M(2,\R)$ with the $l^2$-norm, or with a strictly convex norm. 
Let $\Gamma_0$ be  a finitely generated, nonelementary subgroup of $\SL(2,\R)$, 
 containing $-I$ as the unique element of torsion. 
For all  $\vu \in \R^2\setminus \{0\}$ and all 
continuous functions $f:\R^2\to \R$, we have  
\begin{align*}
   \lim_{T\rightarrow +\infty}\frac{1}{T^{2\delta}} & \sum_{\gamma \in \Gamma_T} f\left(\frac{\gamma \vu}{T} \right) =\\
  & \frac{2}{m^{ps}(T^1S)} 
\int_{\mathcal{D}_0(\vu)} \Theta( \vu , \vv)^\delta \tau(g^{\log \Theta(\vu,\vv)}h^{-\Theta^m(\vu,\vv)}\Psi(\vu)) f(\vv)d\bar{\mu}(\vv),
   \end{align*}
 and the right-hand side is a finite integral, whose total mass does not depend of $\vu$. 

The map $\Theta^m$ equals $0$ in the case of the $l^2$-norm, and is defined in section \ref{applications-Theta} for other norms. 
  \end{thm}

Geometrically, we can rewrite the limit as  
$$
\frac{1}{m^{ps}(T^1S)} \int_{\mathcal{D}_0(\vu)} \mu_{H^-(\Psi(\vu))}\left((h^s \Psi(u)))_{|s+\Theta^m(\vu,\vv)|\le\Theta(\vu,\vv)})\right)\,f(\vv)\,d\bar{\mu}(\vv),
$$
  where $\mu_{H^-}$ is the conditional measure on the strong stable horocyclic foliation
 of $m^{ps}$ (see \S \ref{2}). 
It is remarkable  that the total mass of this integral does not depend on $\vu$, 
whereas the quantity 
$$
\mu_{H^-(\Psi(\vu))}((h^s \Psi(u))_{|s+\Theta^m(\vu,\vv)|\le\Theta(\vu,\vv)}))
$$ does. \\

The assumption that the norm is strictly convex is probably not necessary,
 but guarantees the continuity of $\Theta$.\\

\subsection{Higher dimension and variable curvature}

The extension of the results stated above in higher dimension and/or variable negative curvature is a natural question. 
Let us mention for example articles of Oh-Shah \cite{OS} in higher dimensions,  or Kim \cite{Kim} in complex hyperbolic spaces, where  they study counting results
for  discrete linear orbits. \\

First, observe that the geometric statements on which our distribution results rely extend 
to higher dimension.    
Indeed, we shall prove a higher dimensional version 
(theorem \ref{magic-equivalent-higher-dim}) of
theorem \ref{magic-equivalent}. 

The extension of this theorem in variable negative curvature would work, but in variable negative curvature, 
the amenability of horospherical balls is a problem, and an analogous statement would hold only for 
certain good F\o lner sequences.

For  theorem \ref{integrability-of-balls}, in higher dimension, the same proof would lead to the condition $\delta>2k/3$, 
where $k$ is the maximal rank of the parabolic subgroups of $\Gamma$, whereas
the usual assumption coming from harmonic analysis is $\delta>(n-1)/2$, 
where $n$ is the dimension of the manifold.  
In variable negative curvature, the same proof leads to a similar condition involving the maximal critical exponent of parabolic subgroups of $\Gamma$, 
which is not necessarily equal to $k/2$. \\

However, we decided not to try to 
extend the study of the distribution of nondiscrete orbits of finitely generated groups 
acting linearly on certain linear spaces (higher dimensional, or complex hyperbolic, ...) 
because it would imply a too high technicality of the statements, with a priori the same ideas. \\

The article is organized as follows. Section \ref{2} is devoted to preliminaries on hyperbolic geometry, we prove our 
equidistribution results in section \ref{3}, theorem \ref{th:convexcocompact} and proposition \ref{theresnolimit} 
in section \ref{Relahoro}, geometrically finite surfaces and theorem \ref{almost-sure-cesaro} are studied in section 
\ref{geometrically-finite-groups}, and   theorem \ref{largescale} is proved in the last section.


\section{Preliminaries}\label{2}


\subsection{Action of $\PSL(2,\R)$ on the hyperbolic $2$-dimensional space}

The hyperbolic   upper half plane $\H=\R\times (0,+\infty)$ is
endowed with the hyperbolic metric 
$\frac{dx^2+dy^2}{y}$. 
The group  of isometries preserving orientation of $\H$ 
identifies with $G=\PSL(2,\R)$ acting
by homographies on $\H=\R\times \R_+^*$. An isometry of $G$ 
acts also on $T\H$ and $T^1\H$ via its differential. Moreover, the group $G$ 
acts simply transitively on the unit tangent bundle $T^1\H$, 
so that we identify these two spaces through the  map which sends the unit
vector $(0,1)$ tangent to $\H$ at the origin $o=(0,1)=i$ on the identity element of $G$. 
Let $d$ denote the hyperbolic distance on $\H$, 
and $\partial \H=\R\cup\{\infty\}$ be the boundary at infinity of $\H$. 

The  {\em Busemann cocycle } is the continuous map defined on $\partial \H\times\H^2$ by
$$
\beta_{\xi}(x,y):=\lim_{z\to\xi}\left(d(x,z)-d(y,z)\right)\,.
$$
Define the map
$\displaystyle
v\in T^1\H\mapsto (v^-,v^+,\beta_{v^+}(o,\pi(v))
\in (\partial \H\times \partial \H)\setminus \mbox{Diagonal}\times\R\,,$
where $v^{\pm}$  are the endpoints in 
$\partial \H$ of the geodesic defined by $v$, 
and  $\pi(v)\in\H$ is the basepoint in $S$ of  $v$. 
It defines a homeomorphism between $T^1\H$ and 
$\partial^2\H\times\R:=(\partial \H\times \partial \H)\setminus\mbox{Diagonal}\times\R$, 
and we shall identify 
these two spaces in the sequel. 
An isometry $\gamma \in \PSL(2,\R)$ acts on 
$(\partial \H\times \partial \H)\setminus\mbox{Diagonal}\times\R$ by
$$
\gamma.(v^-,v^+,t)=(\gamma.v^-,\gamma.v^+,t+\beta_{v^+}(\gamma^{-1}.o,o))\,.
$$

Let $\Gamma$ be a discrete subgroup of $G$, without elliptic elements.  
Its {\em limit set} $\Lambda$ is the set  
$\Lambda =\overline{\Gamma.o}\setminus \Gamma.o \subset \partial \H$ 
of the orbit of $\Gamma.o$ in $S^1$ for the usual topology. 
It is also the smallest closed $\Gamma$-invariant subset of $\partial \H$.
 The group $\Gamma$ acts properly discontinuously on the {\em
  ordinary set} $\partial \H\setminus\Lambda$, 
which is a countable union of intervals.

A point $\xi\in\Lambda $ is a  {\em radial} limit point 
 if it is the  limit of a sequence
$(\gamma_n .o)$ of points of $\Gamma . o$ that stay at bounded 
hyperbolic distance 
of the geodesic ray $[o\xi)$ joining $o$ 
to $\xi$.
Let $\Lambda_{\rm rad}$ denote the {\em radial limit set}.

A {\em horocycle} of $\H$ is a euclidean circle tangent to $\partial \H$.
 It can also be defined as a level set of a Busemann function.  
A {\em horoball} is the (euclidean) disc bounded by a horocycle.

An element of $G$ is
 {\em parabolic} if it fixes exactly one point of $S^1$.
Let  $\Lambda_{\rm p}\subset \Lambda$ denote the set of {\em parabolic} 
limit points, that is the points of $\Lambda$ fixed by a parabolic isometry of
 $\Gamma$.

Any hyperbolic surface is the quotient  $S=\Gamma\backslash \H$ of $\H$ by a
 discrete subgroup $\Gamma$ of $G$ without elliptic elements, and its
 unit tangent bundle  $T^1S$ identifies with $\Gamma\backslash G$.

In this article, we always assume $\Gamma$ to be {\em without elliptic elements} and {\em nonelementary},
 that is $\#\Lambda=+\infty$. 
Moreover, we are interested in {\em geometrically finite surfaces} $S$,
 i.e. surfaces whose fundamental group  $\Gamma$ is finitely generated. 
In such cases, the limit set $\Lambda$ is the disjoint union of $\Lambda_{\rm rad}$ 
and $\Lambda_{\rm p}$  \cite{bowditch}.
Moreover, the surface is a disjoint union of a compact part $C_0$, 
finitely many cusps (isometric to $\{z\in\H,\,\, Im\, z\ge cst\}/\{z\mapsto z+1\}$), 
and finitely many 'funnels' 
(isometric to $\{z\in\H,\,\, Re(z)\ge 0,\, 1\le |z|\le a\}/\{z\mapsto az\}=
\{z\in\H,\,\, Re(z)\ge 0\}/\{z\mapsto az\}$, for some $a>1$. 

When $S$ is compact, $\Lambda=\Lambda_{\rm rad}=\partial\H$. 
It is said {\em convex-cocompact}  when it is a geometrically finite surface without cusps. 
In this case, $\Lambda=\Lambda_{rad}$ is strictly included in $\partial \H$
 and $\Gamma$ acts cocompactly 
on  the set 
$(\Lambda\times \Lambda)\setminus \mbox{Diagonal}\times\R\subset T^1\H$. 
When $S$  has finite volume, there are no funnels and 
$\Lambda=\Lambda_{\rm rad}\sqcup\Lambda_p=\partial \H$.


\subsection{Geodesic and horocycle flows}

A hyperbolic geodesic in $\H$ is a vertical line or a half-circle orthogonal to $\partial \H$. 
A vector $v\in T^1\H$ is tangent to a unique geodesic of $\H$. 
Moreover, it is orthogonal to exactly two horocycles passing through its basepoint $\pi(v)$, 
and tangent to $\partial \H$ respectively at $v^+$ and $v^-$. 
The set of vectors $w\in T^1 \H$ such that $w^+=v^+$ and 
based on the same horocycle tangent to $\partial \H$ at $v^+$ 
is the {\em strong stable horocycle} or strong stable manifold 
$W^{ss}(v)\subset T^1 \H$ of $v$. 
The  {\em strong unstable manifold} $W^{su}(v)$ is defined in the same way. 

The {\em geodesic flow}  $(g^t)_{t\in\R}$  acts on $T^1\H$ by moving
 a vector $v$ of a distance $t$ along its geodesic. 
In the identification of
 $T^1\H$ with $\PSL(2,\R)$, this flow corresponds to the right action by
 the one-parameter subgroup 
 $$
 \left\{a_t:=\left(\begin{array}{cc}
e^{t/2} & 0 \\
0 & e^{-t/2} \\
\end{array} \right),\,t\in\R
\right\}.
$$

The {\em strong stable horocyclic flow}  $(h^s)_{s\in\R}$ acts on $T^1\D$ by moving a vector
  $v$ of a distance $|s|$ 
along its strong stable horocycle. There are two possible orientations for this
  flow, 
and we consider the choice corresponding to the right action on $G$
by the one parameter subgroup
 $$ 
\left\{n_s:=\left(\begin{array}{cc}
 1& s \\
0 & 1 \\
\end{array} \right),\,s\in\R
\right\}.$$

For all $s\in\R$ and all $t\in\R$, geodesic and horocyclic flows satisfies
\begin{equation}\label{relationfondamentale}
 g^t\circ h^s=h^{se^{-t}}\circ g^t\,.
 \end{equation}

These two right-actions are well defined on the quotient space 
$T^1 S\simeq \Gamma\backslash G$.

An horocycle of $\H$ or $T^1\H$ is determined by its basepoint 
$\xi\in\partial \H$ and a real parameter, 
the algebraic distance $t=\beta_\xi(o,x)$ between the origin $o$ and the horocycle, where $x$ is any point on $H$. 

The set $\mathcal{H}$ of all horocycles identifies therefore with 
$\partial \H\times \R$, and the 
group $\PSL(2,\R)$ acts naturally on it by 
$\gamma.(\xi,t)=(\gamma.\xi,t+\beta_\xi(\gamma^{-1}.o,o)$.
We refer to section \S \ref{Relahoro} for an explicit 
identification of $\mathcal{H}=\partial \H\times \R$ with 
$\R^2\setminus\{0\}/\pm$, where the action by isometries of $\PSL(2,\R)$ 
on $\mathcal{H}$ corresponds to the linear action of $\PSL(2,\R)$ on $\R^2\setminus\{0\}/\pm$. 
The {\em non-wandering set for the horocyclic flow} $\mathcal{E}$ is 
the set $\mathcal{E}=\Gamma\backslash ((\partial \H \times \Lambda)\setminus Diagonal \times \R)$. 
It is known \cite{Dalbo} that a vector $u \in \mathcal{E}$ 
is either periodic, or its orbit under the horocyclic flow is dense in $\mathcal{E}$.


\subsection{The Patterson-Sullivan construction}

Let $\delta$ be the critical exponent of $\Gamma$, defined by
 $$
 \delta:=\limsup_{T\to\infty}\frac{1}{T}\log\#\{\gamma\in \Gamma, d(o,\Gamma.o)\le T\}\,. 
 $$
The well known Patterson construction provides a {\em conformal density}
 of exponent $\delta$ on $\partial \H$, 
that is a collection $(\nu_x)_{x\in\H}$ of  measures, 
supported on $\Lambda\subset \partial \H$, s.t. $\nu_o(\partial \H)=1$,
 $\gamma_*\nu_x=\nu_{\gamma.x}$ for all $\gamma\in\Gamma$, 
and 
$$
\frac{d\nu_x}{d\nu_y}(\xi)=e^{-\delta \beta_{\xi}(x,y)}.
$$ 

The Bowen-Margulis or Patterson-Sullivan measure $m^{ps}$ on  $T^1 S$  
is defined locally as the product 
$$
dm^{ps}(v)=
\exp\left(\delta\beta_{v^-}(o,\pi(v))+\delta\beta_{v^+}(o,\pi(v))\right)
d\nu_o(v^-)d\nu_o(v^+)dt
$$
in the coordinates 
$\Omega\simeq \Gamma\backslash (\Lambda_{\Gamma}^2\setminus{\rm Diagonal} \times\R)$.

Under our assumptions on $S$, it is well known \cite{Sullivan} 
that the Bowen-Margulis measure is $(g^t)$-invariant,  finite and ergodic,
that there exists a unique conformal density of exponent $\delta$,
 that all measures $\nu_x$ are nonatomic and give full measure to the radial limit set. 
Moreover, the Bowen-Margulis-Patterson-Sullivan measure is the measure 
of maximal entropy of the geodesic flow,
 and it is fully supported on the nonwandering set $\Omega$ of the geodesic flow.
Note that in general, this measure is {\em not} invariant under the horocyclic flow, except on 
finite volume surfaces, where $\Lambda=\partial \H$, $\Omega=\mathcal{E}=T^1S$, $\delta=1$
and $\nu_o$ is the Lebesgue visual measure on $\partial \H$. 

The Patterson-Sullivan measure induces a 
$\Gamma$-invariant measure on the space of horocycles, 
defined by
$d\hat{\mu}(\xi,t)=\exp(\delta t)d\nu_o(\xi)dt$. 
Roblin \cite{Roblin} proved that it is the only $\Gamma$-invariant, ergodic measure 
with full support 
in $\Lambda\times \R\subset\mathcal{H}$. 


\subsection{Foliations and the Burger-Roblin measure}

The orbits of the horocyclic flow on $T^1S$ (resp. $T^1\H$) form a one-dimensional foliation 
$\mathcal{W}^{ss}$ of $T^1S$ (resp. of $T^1\H$). We denote by $H^-$, 
or $H^-(u)=\{h^su,\,s\in\R\}$ a leaf of any of these two foliations. 
Given a chart of the foliation (or flow box) $\varphi:B\to \R^2\times\R$, we write $B=T\times P$, 
where $T=\varphi^{-1}(\R^2\times\{t\})$ is a transversal, and $P=\varphi^{-1}(\{x\}\times \R)$ is a
plaque. 
On $T^1\H$, the set $\mathcal{H}=\partial\H\times\R$ of horocycles provides a natural global transversal 
to $\mathcal{W}^{ss}$. 

The conditional measures of $m^{ps}$ on stable horocycles are defined by 
$$
d\mu_{H^-(v)}(v)=e^{\delta\beta_{v^-}(o,\pi(v)}\,d\nu_o(v^-),
$$
(the formula  is independant of the choice of $o$), so that locally, if $f:T^1\H\to \R$ is 
continuous with compact support, one has 
$\int_{T^1\H} f\,d\tilde{m}^{ps}=\int_{\mathcal{H}}\left(\int_{H^-(v)} f d\mu_{H^-(v)}\right) d\hat{\mu}(v^+,t)$.

The measures $(\mu_{H^-(v)})_{v\in T^1\H}$ are well defined on the quotient on $T^1S$, but on 
$T^1S$, there is no global transversal to the foliation. 
One has to consider transverse invariant measures, that is a collection $\nu=\{\nu_T\}$ of 
measures on all transversals $T$, invariant by the holonomies, that is homeomorphisms $\zeta:T\to T'$ 
which follow the leafs 
between two transversals of a same box. 

The unique $\Gamma$-invariant, ergodic measure $\hat{\mu}$ of full support on $\Lambda\times \R\subset \mathcal{H}$ 
 induces on the quotient a transverse invariant measure $\nu=\{\nu_T\}$, which is the unique (up to normalization)
transverse invariant measure with full support in the nonwandering set $\mathcal{E}$ of the horocyclic flow. 
Locally, in a box
$B=T\times P$, $\int_B f dm^{ps}=\int_T\left( \int_{\{t\}\times P} fd\mu_{H^{-}(v)} \right) d\nu_T$. 

Denote also by $(\lambda_{H^-(v)})$ the collection of Lebesgue measures on all horocycles associated 
with the parametrization of the horocyclic flow. 

Introduce the measure $m$ defined locally, for $f:B=T\times P\to \R$ continuous with compact support,
 as the (noncommutative) product
$$
\int_B f \, dm = \int_T \left(\int_{\{t\}\times P} f \,d\lambda_{H^-} \right) d\nu_T\,.
$$
One can now reformulate Roblin's result as follows. 
On geometrically finite surfaces, except the  probability measures supported on  periodic horocycles, 
and the infinite measures supported on  wandering horocycles, 
{\em the measure $m$  is the unique (up to normalization)
 ergodic invariant measure fully supported in the
nonwandering set} 
$
{\mathcal
  E}\simeq \Gamma\backslash \left((\Lambda\times S^1)
\setminus\mbox{Diagonal}\times \R\right)$ of $(h^s)_{s\in\R}$. 
It is an infinite locally finite measure. 

Its lift to $T^1\H$, still denoted by $m$, can be understood 
  in the $\mathcal{H}\times \R$ decomposition as follows.
 If $f:T^1\H\to \R$ is continuous with compact support, 
$$
\int_{T^1\H}f\,dm=\int_{\mathcal{H}} \left(\int_{H^{-}(v)}f d\lambda_{H^-}\right)d\hat{\mu}(v^+,t) .
$$

Let us mention that the two families 
$(\lambda_{H^-(v)})$ and $(\mu_{H^-(v)})$ vary
continuously when $v$ moves transversally to the leaves: 
for all boxes $B\subset T^1S$ and 
continuous maps  $f:B\to \R$ with compact support, 
the two following maps are continuous 
$$
t\in T\mapsto \int_{\{t\}\times P} f(v)\,d\mu_{H^-(v)}\quad\mbox{and}\quad
t\in T\mapsto \int_{\{t\}\times P} f(v)\,d\lambda_{H^-(v)}
$$

By construction, the measure $m$ is quasi-invariant under the action of the geodesic 
flow, and more precisely: 
$$
\frac{dg^t_*m}{dm}(v)=e^{(1-\delta)t}\,,
\quad\mbox{ for all }t\in\R\quad\mbox{ and }v\in\mathcal{E}\,.
$$




\section{Equidistribution of horocycles}\label{3}

Define $\displaystyle \tau(u)=\mu_{H^-}((h^s u)_{|s|\le 1}).$
Using the definition of $\mu_{H^-}$, one gets easily the useful relation 
$$
\mu_{H^-}((h^s u)_{|s|\le R})=R^\delta\tau(g^{\log R} u)\,.
$$


\subsection{Higher dimensional equidistribution}

The formalism of foliations allows to work with higher dimensional manifolds. 
We will introduce some additional notations, and prove theorem \ref{magic-equivalent}
and a higher-dimensional analoguous statement together. 

The hyperbolic space $\H^n$, $n\ge 2$, identifies with $\R^{n-1}\times \R_+^*$, with 
the hyperbolic metric $\frac{dx_1\dots dx_n}{x_n}$. 
A horosphere is a horizontal hyperplane or a sphere tangent to $\R^{n-1}\times \{0\}$. 
The space $\mathcal{H}$ of horospheres still identifies 
with $\partial \H^n\times \R$, where $\partial \H^n=\R^{n-1}\cup\{\infty\}\simeq S^{n-1}$. 

If $\Gamma$ is a discrete nonelementary group of isometries of $\H^n$, let $M=\Gamma\backslash \H^n$, and 
$T^1M=\Gamma\backslash \PSL(2,\R)$. 

We endow all  horocycles $H=(\xi,t)$ with the distance
induced by the induced riemannian metric on the horosphere. 
It corresponds to the classical euclidean distance on the 
horizontal horosphere $H(\infty,0)$ at euclidean height $1$. 
This distance lifts to a distance $d_{H^-}$ on all strong stable horospheres, which satisfies 
$ d_{H^-}(g^t v,g^t w)=e^{-t}d_{H^-}(v,w)$ for all $v,w$ on a same horosphere, and $t\in\R$.
We denote by $B_{H^-}(u,r)$ the ball of radius $r$ for this distance. 
We still have 
$$
\tau(u)=\mu_{H^-}(B_{H^-}(u,1))\quad\mbox{and}\quad \tau(g^t u)=e^{-\delta t}\mu_{H^-}(B_{H^-}(u,e^t))\,.
$$

Denote by $(\lambda_{H^-})$ the family of Lebesgue measures on strong 
stable horospheres of $T^1M$ or $T^1\H^n$,
 normalized  in such a way that on the horizontal horosphere 
of the unit vector $(0,\dots, 0,1)$ with base point $(0,\dots, 0,1)$,
 it coincides with the usual Lebesgue measure, and that it satisfies 
$g^t_*\lambda_{H^-(v)}= e^{t} \lambda_{H^-(g^t v)}$. 

We still denote by $m$ the measure obtained locally
as the (noncommutative) product of $\hat{\mu}$ and $(\lambda_{H^-})$ on 
$T^1M$, or equivalently as the noncommutative product of the unique
 transverse measure $\nu=(\nu_T)$ of full support 
in $\mathcal{E}$ with $(\lambda_{H^-})$, on $T^1M$.

With these notations, we have 

\begin{thm}\label{magic-equivalent-higher-dim} 
Let $\Gamma$ be a discrete nonelementary convex-cocompact group 
of isometries of $\H^n$ without elliptic elements, and $M=\Gamma \backslash\H^n$. 
Let $u\in T^1M$ and $r_n\to +\infty$ be a sequence such that 
 the Lebesgue measures on the leaves satisfy the following amenability condition:
there exist a family of  boxes $B\subset T^1M$ covering $\Omega$, s.t. $m^{ps}(B)>0$, and s.t. for some $r_0=r_0(B)>0$, 
we have 
\begin{equation}\label{amenability-lebesgue}
\frac{\lambda_{H^-}(B\cap \left(B_{H^-}(u,r_n+r_0)\setminus B_{H^-}(u,r_n-r_0)\right))}{\lambda_{H^-}(B\cap B_{H^-}(u,r_n))}\to 0\quad\mbox{when}\quad r_n\to +\infty\,.
\end{equation}
Then, for all $f:T^1M\to \R$ continuous with compact support, we have:
$$
\lim_{r_n\to +\infty}\frac{1}{r_n^\delta \tau(g^{\log r_n}u)}\int_{B_{H^-}(u,r_n)}f \,d\lambda_{H^-(u)}=
\frac{1}{m^{ps}(T^1M)}\int_{T^1M} f\,dm\,.
$$
\end{thm}

For a fixed box $B$, the sequence $\left(\lambda_{H^-}(B\cap B_{H^-}(u,r))\right)$
is bounded by  $ \lambda_{H^-}(B_{H^-}(u,r))\sim c r^{d-1}$, where $d=\dim M$,  so that it 
grows polynomially with $r$. As $\Gamma$ is convex-cocompact, $\Omega$ can be covered by finitely 
many boxes $B$ such that $m^{ps}(B)>0$, so that for "many" sequences $r_n\rightarrow +\infty$, 
the assumption (\ref{amenability-lebesgue}) is satisfied.

The fact that this theorem is stated in the convex-cocompact case, 
and not in the geometrically finite setting is also due
to a problem of amenability of horospherical balls, but with respect to the measure $\mu_{H^-}$ induced 
by the Patterson-Sullivan measure $m^{ps}$ on the horospheres, 
which will be described in the next section. 
But we will prove it essentially by the same arguments than theorem \ref{magic-equivalent}. 



\subsection{Higher-dimensional equidistribution towards $m^{ps}$}

Theorems \ref{magic-equivalent} and \ref{magic-equivalent-higher-dim} follow from another
equidistribution result of certain horospherical averages towards 
the Bowen-Margulis-Patterson-Sullivan measure $m^{ps}$.  
For $f:T^1 M\to \R$ continuous with compact support, define 
$$
M_{r,u}(f)=\frac{1}{\mu_{H^-}(B_{H^-}(u,r))}\int_{B_{H^-}(u,r)} f\,d\mu_{H^-}\,.
$$

 In \cite{Scha2} were proved several equidistribution results,
 among them theorem \ref{equidistribution-PS} below,
and a ratio equidistribution theorem towards the measure $m$, 
whose proof implied in fact implicitely theorem \ref{magic-equivalent}. 

\begin{thm}\label{equidistribution-PS} Let $M=\Gamma\backslash\H^n$ 
be a nonelementary geometrically finite hyperbolic manifold. 
Assume either that $n=2$ and $M$ is a surface, or that 
$M$ is convex-cocompact, or that for some $r_0>0$ and all $u\in \Omega$, we  have
\begin{equation}\label{Folner}
\frac{\mu_{H^-}(B_{H^-}(u,r+r_0)\setminus B_{H^-}(u,r-r_0))}{\mu_{H^-}(B_{H^-}(u,r))}
\to 0 \quad \mbox{when} \quad r\to +\infty\,.
\end{equation}
Let $u\in \mathcal{E}$ be a non-periodic vector for the horocyclic flow.
Then, the sequence $(M_{R,u})_{R>0}$ 
converges weakly to the normalized Patterson-Sullivan measure
 $\frac{m^{ps}}{m^{ps}(T^1S)}$ when $R\to +\infty$.
In other words, for all continuous maps $f:T^1M\to  \R$ with compact support, 
$$
M_{r,u}(f)\to \frac{1}{m^{ps}(T^1S)}\int_{T^1S}f\,dm^{ps}\quad\mbox{when}\quad r\to +\infty\,.
$$
Moreover, when $M$ is convex-cocompact, this convergence is uniform in $u\in  \Omega$. 
\end{thm}

\begin{proof}
In the case $n=2$, this is \cite{Scha2}, thm 1.2.
The uniform convergence in the convex-cocompact case is due to Roblin \cite{Roblin} (see also \cite[thm 3.1]{Scha2}). 
Under the assumption that (\ref{Folner}) is true, 
this theorem follows from \cite{Scha0}, thm 8.1.1. or \cite{Scha2}, remark 
3.8. It remains to prove that (\ref{Folner}) 
is true on hyperbolic convex-cocompact manifolds. 

The proposition page 1799 in section 3.1 of \cite{Roblin2} implies 
that (even in the geometrically finite case)
for all $u\in \Omega$, and $r>0$, $\mu_{H^-}(\partial B_{H^-}(u,r))=0$. 
In the convex-cocompact setting, observe that
\begin{eqnarray*}
\frac{\mu_{H^-}(B_{H^-}(u,r+r_0)\setminus B_{H^-}(u,r-r_0))}{\mu_{H^-}(B_{H^-}(u,r))}=\\
\frac{\mu_{H^-}(B_{H^-}(g^{\log r}u,1+\frac{r_0}{r})\setminus B_{H^-}(g^{\log r}u,1-\frac{r_0}{r}))}{\mu_{H^-}(B_{H^-}(g^{\log r}u,1))}\,.
\end{eqnarray*}
Let us prove that this quantity  converges to $0$ when $r\to +\infty$. 
Let $\alpha\ge0$ be a limit point of this ratio when 
$r_n\to +\infty$.
When $\Gamma$ is convex-cocompact, the set $\Omega$ is compact, 
and the limit points of $(g^tu)$ when $t\to +\infty$ are in $\Omega$,  so that up to a subsequence, 
if $u\in \mathcal{E}$, we can assume that $g^{\log r_n}u$ converges to some $v\in \Omega$. 
But $\mu_{H^-}(\partial B_{H^-}(v,1))=0$, so that the ratio 
$\frac{\mu_{H^-}(B_{H^-}(g^{\log r_n}u,1+\frac{r_0}{r_n})\setminus B_{H^-}(g^{\log r_n}u,1-\frac{r_0}{r_n}))}{\mu_{H^-}(B_{H^-}(g^{\log r_n}u,1))}$
 has to converge to $0$ when $r_n\to \infty$.
 Thus, $\alpha=0$. 

The conclusion of the theorem follows. 
\end{proof}


\subsection{Proofs of theorems \ref{magic-equivalent} and \ref{magic-equivalent-higher-dim}}

In this section, we work on a nonelementary geometrically finite surface $S=\Gamma\backslash \H$, or
on a higher dimensional nonelementary convex-cocompact manifold $M=\Gamma\backslash \H^n$.

 Note  the following property of $\tau$.

\begin{lem}[Schapira  \cite{Scha2}, Fact 3.7]\label{propertiesoftau}
For any $u\in \mathcal{E}$ such that $u^+\in \Lambda\setminus\Lambda_p$ is a nonparabolic limit point,
$$
\lim_{t\rightarrow +\infty} t^\delta \tau(g^{\log T}u) =\lim_{t\to +\infty}\mu_{H^-}(B_{H^-}(u,t))= +\infty.
$$
\end{lem}

Let us begin the proof of theorems \ref{magic-equivalent} 
and \ref{magic-equivalent-higher-dim}. 

\begin{proof} The strategy of the proof follows the proof of 
theorem 1.2 in \cite{Scha2}, and differs only at the end. 
 It is enough to prove that for $u\in \mathcal{E}$,  $f:T^1M\to \R$ 
continuous with small compact support 
included in a relatively compact box $B=T\times P$,   we have
$$
\frac{\int_{B_{H^-}(u,r)}f \,d\lambda_{H^-}}{\mu_{H^-}(B_{H^-}(u,r))}\to
 \frac{\int_{T^1M} f\,dm}{m^{ps}(T^1M)}\,.
$$
One can assume that $m(B)>0$ and $m(\partial B)=0$. 

$*$ {\em First step : Assume also that $m^{ps}(B)>0$. }\\
Therefore, up to shrinking $B=T\times P$ a little bit, we can assume that
for all $v\in B$, if $P_v$ is the plaque of $B$ containing $v$, 
$\lambda_{H^-}(P_v)\ge l_0>0$ and $\mu_{H^-}(P_v)\ge m_0>0$. 
Let $r_0=\sup_{v\in B} \sup_{w\in P_v} d_{H^-}(v,w)$. 

On a transversal $T\subset B$, define 
$$
\nu_T^{r,u}=\frac{1}{\mu_{H^-}(B_{H^-}(u,r))}\sum_{t\in T\cap B_{H^-}(u,r)}\delta_t\,,\quad \mbox{and}\quad 
\nu_T^{B,r,u}=\frac{\mu_{H^-}(B_{H^-}(u,r))}{\lambda_{H^-}(B\cap B_{H^-}(u,r))}\,\nu_{T}^{r,u}\,.
$$

Now, observe that 
$$
M_{r,u}(f):=\int_T\left(\int_{P_t} f\,d\mu_{H^-}\right)d\nu_T^{r,u} +R_{r,u}(f)\,, \quad\mbox{and} $$
$$
M_{r,u}^{B,\lambda}(f):=\frac{ \int_{B_{H^-}(u,r)} f\,d\lambda_{H^-}(v)}{\lambda_{H^-}(B\cap B_{H^-}(u,r))}=
\int_T\left(\int_{P_t} f\,d\lambda_{H^-}\right)d\nu_T^{B,r,u} +R_{B,r,u}(f)\,, 
$$
where the error terms $R_{r,u}(f)$ and $R_{B,r,u}(f)$ 
correspond on one hand to pieces of $B_{H^-}(u,r)$ 
such that $B_{H^-}(u,r)$ intersects the plaque $P_v$ of $v$ 
without intersecting $T$, and on the other
hand to pieces such that $B_{H^-}(u,r)$ intersects $P_v$ and $T$, 
but in such a way that 
$B_{H^-}(u,r)\cap P_v\varsubsetneq P_v$. We refer to \cite{Scha2},
 figure 3.1 and proof of theorem 1.2 for more details. 

In the case of surfaces, there are at most two plaques corresponding to these error terms, so that 
we can bound them   as follows 
$$
R_{r,u}(f)\le \frac{2\|f\|_\infty\sup_{t\in T}\mu_{H^-}(P_t)}{\mu_{H^-}(B_{H^-}(u,r))}\quad \mbox{and}\quad
R_{B,r,u}(f)\le 
\frac{2\|f\|_\infty\sup_{t\in T}\lambda_{H^-}(P_t)}{\lambda_{H^-}(B_{H^-}(u,r)\cap B)}\,.
$$
They converge to $0$ when $r\to +\infty$ as the denominators converge to $+\infty$. 

In the higher dimensional case, if $r_0\ge\sup_{v\in B}\sup_{w\in P_v}d_{H^-}(v,w)$, we can write 
$$
R_{r,u}(f)\le 
\frac{\mu_{H^-}(B_{H^-}(u,r+r_0))\setminus B_{H^-}(u,r-r_0))}{\mu_{H^-}(B_{H^-}(u,r))}\times \|f\|_\infty\,,\quad\mbox{and}
$$
$$
R_{B,r,u}(f)\le 
\frac{\lambda_{H^-}(B\cap B_{H^-}(u,r+r_0))\setminus B_{H^-}(u,r-r_0))}{\lambda_{H^-}(B\cap B_{H^-}(u,r))}\times \|f\|_\infty\,.
$$
Therefore, as we work under the assumptions of theorem  \ref{equidistribution-PS}, and thanks to the assumption
(\ref{amenability-lebesgue}), 
these two error terms 
converge to $0$ as $r_n\to \infty$.  

Now, by theorem \ref{equidistribution-PS}, $M_{r,u}$ converges weakly 
to $\frac{1}{m^{ps}(T^1M)}m^{ps}$ and $R_{r,u}(f)\to 0$, so that the transverse measure $\nu_T^{r,u}$
converges to the transverse measure $\nu_T$ induced by the Patterson-Sullivan measure. 

Similarly, as $M_{r,u}^{B,\lambda}$ defines a probability
measure on $B$, when $r_n\to +\infty$, all its limit points for the 
weak topology on $B$ are probability measures on $B$. 
As $R_{B,r_n,u}(f)\to 0$, 
we deduce therefore that $\nu_T^{B,r_n,u}$ has limit points 
for the weak topology on $T$ when $r_n\to +\infty$, 
which define transverse invariant measures on all 
transversals of $B$, and the limit points of the sequence
$(M_{r_n,u}^{B,\lambda})$ can be written as the noncommutative product of these 
limit transverse measures by $(\lambda_{H^-})$. 
 
As $\nu_T^{B,r,u}= \frac{\mu_{H^-}(B_{H^-}(u,r))}{\lambda_{H^-}(B\cap B_{H^-}(u,r))}\,\nu_{T}^{r,u}$, 
and $\nu_T^{r,u}\to \nu_T$, we deduce that all limit points of the ratio 
$\frac{\mu_{H^-}(B_{H^-}(u,r_n))}{\lambda_{H^-}(B\cap B_{H^-}(u,r_n))}$ when $r_n\to +\infty$ 
are positive and finite, and that all limit points of $(\nu_T^{B,r_n,u})$ are proportional to $\nu_T$. 

This implies that all limit points of $M_{r_n,u}^{B,\lambda}$
are probability measures on $B$ 
proportional to the Burger-Roblin measure $m$ restricted to $B$. 
All these measures being probability measures giving mass $1$ to $B$, 
they are all equal, so that we proved the weak convergence on $B$ 
of $M_{r_n,u}^{B,\lambda}$ towards $\frac{1}{m(B)}m_{|B}$, and as a consequence, the convergence of 
the ratio $\frac{\mu_{H^-}(B_{H^-}(u,r_n))}{\lambda_{H^-}(B\cap B_{H^-}(u,r_n))}$
to a positive finite constant $c$. 
A normalization argument implies now the following convergence: 
$$ 
\frac{\mu_{H^-}(B_{H^-}(u,r_n))}{\lambda_{H^-}(B\cap B_{H^-}(u,r_n))}\to 
\frac{m^{ps}(T^1M)}{m(B)}\quad\mbox{when}\quad r_n\to +\infty\,.
$$
This gives immediately, for $f:B\to \R$ continuous with compact support,
 the desired convergence: 
$$
\frac{\int_{B_{H^-}(u,r_n)} f\,d\lambda_{H^-}}{\mu_{H^-}(B_{H^-}(u,r_n))}\to
 \frac{\int_{T^1M} f\,dm}{m^{ps}(T^1M)}
\quad\mbox{when}\quad r\to +\infty\,.
$$

$*$ {\em Second  step: consider a relatively compact box $B$ such that $m(B)>0$ but $m^{ps}(B)=0$. }
Write as before $B=T\times P$. 
The only problem on $B$ comes from the fact that 
we cannot use directly theorem \ref{equidistribution-PS} to deduce the convergence of $\nu_T^{r,u}$ to 
$\nu_T$, because $m^{ps}(B)=0$. 

However, as the support $\mathcal{E}$ of $m$ is the union of horospheres 
intersecting $\Omega$, which is the support 
of $m^{ps}$, if $B$ is small enough, it is possible to find a holonomy
 map $\zeta:T\to T'$ along the leaves of 
the strong stable foliation, and another box $B'=T'\times P'$, such that $m^{ps}(B')>0$ and $m(B')>0$. 
Now, the above reasoning applies on $B'$ and  implies that $\nu_{T'}^{r,u} \to \nu_{T'}$. 
On the transversal $T$ of $B$, the two measures $\zeta^{-1}_*\nu_{T'}^{r,u}$ and $\nu_T^{r,u}$ differ from a quantity
bounded by $\frac{\mu_{H^-}(B_{H^-}(u,r+r_0)\setminus B_{H^-}(u,r-r_0))}{\mu_{H^-}(B_{H^-}(u,r))}$, 
with $r_0=\sup_{t\in T} d_{H^-}(t,\zeta(t))$,  which goes to $0$ as $r\to +\infty$. 
As the transverse measure $(\nu_T)_T$ is invariant under the holonomy of the foliation, we deduce that
$\nu_{T}^{r,u}$ converges to $\frac{1}{m^{ps}(T^1S)}\nu_T$ as $r\to +\infty$. 

The end of the proof is as above.\\

$*$ {\em  Third step: proof of the uniform convergence in theorem \ref{magic-equivalent} in  the case of a convex-cocompact surface} 
Fix a continuous map $f$ with compact support. It can be shown that the maps $u\to \int_{-1}^1f\circ g^{-t}(h^su)\,ds$
 are equicontinuous in $t\ge 0$. It is proven for example 
in \cite[lemma 5.10]{Babillot}, apparently under the assumption that $\Gamma$ is a lattice, but this assumption is not used in the proof.

As $\Omega$ is compact, they are uniformly equicontinuous on $\Omega$.
As the maps $u\mapsto \tau(u)=\mu^{ps}((h^su)_{|s|\le 1})$ are continuous on $\Omega$, and 
therefore uniformly continuous, we deduce that the ratios 
$$
u\mapsto Q(f,u,t):=t^{1-\delta}\frac{\int_{-1}^1 f\circ g^{-\log t}(h^sv)\,ds}{\mu^{ps}_{H^-(v)}((h^sv)_{|s|\le 1})}
$$
are multiplicatively equicontinuous in $t\ge 1$, and therefore uniformly multiplicatively equicontinuous in $u\in \Omega$, for $t\ge 1$, in the sense that 
given $\varepsilon>0$, there exists $\eta>0$, such that if $v,w\in \Omega$, with $d(v,w)<\eta$, then for all $t\ge 1$, 
$\displaystyle \left|\frac{Q(f,v,t)}{Q(f,w,t)}\right|\le e^\varepsilon$.

By theorem \ref{magic-equivalent-flowed},  these ratios converge, for all fixed $u\in\Omega$, to $\frac{1}{m^{ps}(T^1M)}\int_{T^1M}f\,dm$ when 
$t\to +\infty$. Thus,  the uniform equicontinuity and a standard argument imply that this convergence is uniform in $u\in\Omega$. 


Recall now that
$$
Q(f,v,t)=\frac{\int_{-t}^t f\circ h^s(g^{-\log t}v)\,ds}{t^\delta \tau(g^{\log t}g^{-\log t} v)}\,.
$$
In particular, we deduce from the uniform convergence of $Q(v,f,t)$ that the maps 
$$
w\mapsto \frac{\int_{-t}^t f\circ h^s (w)\,ds}{t^\delta\tau(g^{\log t}w)} 
$$
converge also uniformly to to $\frac{1}{m^{ps}(T^1M)}\int_{T^1M}f\,dm$ when 
$t\to +\infty$.
\end{proof}


\subsection{Equidistribution of horocycles pushed by the geodesic flow}

In the proof of theorem \ref{largescale}, we will need another
 equidistribution theorem for horocycles, 
due to Thomas Roblin \cite{Roblin}, thm 3.4 (see also \cite{mbab2}). 
In fact, the statement of theorem 3.4 of Roblin seems slightly different, 
but in its proof, page 52, formula $(*)$, he establishes exactly the result below. 

\begin{thm}[ Roblin \cite{Roblin} ]\label{magic-equivalent-flowed} 
Let $S$ and $\tau$ be as in Theorem \ref{magic-equivalent}. 
For any $u\in T^1S$, and all $f:T^1S\to \R$ continuous with compact support, we have
\begin{multline*}
\lim_{t\rightarrow +\infty} \frac{1}{t^{\delta}}\int_{-t}^t f(h^sg^{-\log t}u)ds = 
\lim_{t\rightarrow +\infty} t^{1-\delta}\int_{-1}^1 f(g^{-\log t}h^su)ds  \\
 = \frac{\tau(u)}{m^{ps}(T^1S)}\int_{T^1S}f\,dm \,. 
\end{multline*}
\end{thm}  

Note that more generally, Roblin obtains a result in all dimensions. 
If $M$ is a geometrically finite manifold
of any dimension, we have
$$
\lim_{t\to +\infty} \frac{1}{t^\delta}\int_{B_{H^-}(g^{-\log t}u,t)}f(v)\,d\lambda_{H^-}(v)
=
\frac{\tau(u)}{m^{ps}(T^1M)}\int_{T^1M}f\,dm\,.
$$

 Note that a similar statement was also proved by Oh and Shah \cite{OS}. In fact, it is possible to  deduce 
Theorem \ref{magic-equivalent-flowed} from the Theorem of \cite[Thm 3]{mbab2} below exactly in the same way 
than to deduce  Theorem \ref{magic-equivalent} from Theorem \ref{equidistribution-PS}. 
 
\begin{thm}[Babillot] For all $u\in T^1M$ such that $\tau(u)>0$, 
and all continuous $f$ with compact support on $T^1M$, we have
\begin{multline*}
\lim_{t\rightarrow +\infty} \frac{1}{\mu_{H^-(g^{-\log t}u)}(B_{H⁻}(g^{-\log t}u,t))} \int_{B_{H⁻}(g^{-\log t}u,t)} f(v)d\mu_{H^-}(v)=\\
\frac{1}{m^{ps}(T^1M)}\int_{T^1M}f dm^{ps}.
\end{multline*}
\end{thm}


\section{From the plane to the space of Horocycles}
\label{Relahoro}

Let us describe the relation between the linear action of a discrete subgroup $\Gamma_0$ 
of $G_0=\SL(2,\R)$ on $\R^2$ and the action of the horocyclic flow on the quotient space 
$\Gamma\backslash \PSL(2,\R)$, where $\Gamma$ is the image of $\Gamma_0$ for the quotient map
 $G_0\rightarrow G$. Since we assume that $-I \in\Gamma_0$, $\Gamma_0$ is also the preimage of $\Gamma$.


\subsection{Linear action of $\SL(2,\R)$ on $\R^2\setminus\{0\}$} \label{defpsi}
The relation described below was already in \cite{led}.
We will borrow notations and results of \cite{MW}, 
where the reader will find a more detailled presentation of the following objects, 
and proofs of their properties.
 Fix $\vu_0=\left(\begin{array}{c} 1 \\ 0 \end{array} \right)$.

Recall that $\SL(2,\R)$ is diffeomorphic to $K\times A\times N$, 
where $K=SO(2,\R)$, $A=\{\left(\begin{array}{cc} e^t & 0 \\ 0 & e^{-t} \end{array}\right),\,t\in\R\}$
and $N=\{\left(\begin{array}{cc} 1 & s \\ 0 & 1 \end{array}\right),\,s\in\R\}$, 
and that the stabilizer of $\vu_0$ in $\SL(2,\R)$ is equal to $N$. 
Thus, there is a natural identification of $\R^2\setminus \{0\}$ with $\SL(2,\R)/N$, 
and a natural section $\Psi:\R^2\setminus\{0\} \to \SL(2,\R)$
which associates to $\vu=k.a.\vu_0$ the matrix $k.a\in \SL(2,\R)$.
 It can be written explicitely as 
$$
\Psi(\vu)=
\left(\begin{array}{cc} \vu_x & -\vu_y/(\vu_x^2+\vu_y^2) \\ \vu_y & \vu_x/(\vu_x^2+\vu_y^2) \end{array} \right).
$$
This is a continuous section, in the sense that 
 for any $\vv\in \R^2$, $\vv=\Psi(\vv)\vu_0$.

 The section $\Psi$ satisfies
 \eq{propPsi}{
 \Psi(e^t\vv)=\Psi(\vv)a_{2t}.
}

Geometrically, in the natural identification of  $T^1\H$ with $\PSL(2,\R)$ which associates to $g\in \PSL(2,\R)$ 
the vector $g.u_0$, where $u_0$ is the unit vector $i$ with basepoint $i$, if $\vu=r_\theta. \vu_0$ has coordinates 
$(\cos\theta,\sin\theta)$,  
the image of $\Psi(\vu)$ in $T^1\H=\PSL(2,\R)$ 
is the unit vector with base point $i$ and with coordinates $(\sin 2\theta, \cos 2\theta)$.
If $\vu=r_\theta.a_t.\vu_0$, the image of $\Psi(\vu)$ in $T^1\H$ is the image 
of $\Psi(r_\theta.\vu_0)$ under the geodesic flow $g^{2t}$.


\subsubsection*{Identification of $\mathcal{H}$ with $\R^2\setminus\{0\}/\pm$}

Define a map from $\R^2\setminus\{0\}$ to $\mathcal{H}$ by
$$
\vv \in\R^2\setminus\{0\} \to (\Psi(\vv)(\infty),2 \log |\vv|)\in S^1\times \R\simeq \mathcal{H}\,,
$$ 
where $|.|$ is the standard euclidean norm on $\R^2$. 
This map is $\SL(2,\R)$-equivariant, for the linear action of $\SL(2,\R)$ on $\R^2\setminus\{0\}$, and 
the natural action of $\SL(2,\R)$ on $\mathcal{H}$ induced by the left-multiplication on $\PSL(2,\R)$, 
as described in \S \ref{2}. 
Moreover, this map induces a bijective map from $\R^2\setminus\{0\}/\pm$ 
to $\mathcal{H}$, whose inverse is
$$
\Phi: (\xi,t)\mapsto  \pm(e^{t/2}\sin(\theta),e^{t/2}\cos(\theta))\,,
$$ 
 where $\theta$ is such that 
$$
\left(\begin{array}{cc} \sin(\theta) & -\cos(\theta)\\
\cos(\theta)& \sin(\theta) \end{array}\right) (\infty)=\xi,
$$
or equivalently $\tan(\theta)=\xi \in \partial \H = \R\cup\{\infty\}$. 
These maps allow us to identify $\mathcal{H}$ with $\R^2\setminus\{0\}/\pm$.

\subsubsection*{Duality on measures}
Now, we want to use this identification to describe the $\Gamma$-invariant measure $\hat{\mu}$ on $\mathcal{H}$ 
as a measure $\bar{\mu}$ on $\R^2\setminus\{0\}$ in polar coordinates. 
Write $\R^2\setminus\{0\}$ as $S^1\times \R_+^*$. 
There is a unique symmetric probability on $S^1$, denoted by  $\bar{\nu}_o$, whose image on $\partial \H$ is $\nu_o$, 
simply obtained by giving equal weights locally to the 2 sheets of the covering map $S^1\rightarrow \partial \H$, $(\sin \theta,\cos \theta)\mapsto \tan \theta$. 

Thus, via the bijection $\Phi$ between $\mathcal{H}$ and $\R^2\setminus\{0\}/\pm$, 
a symmetric continuous function $f$ on $\R^2\setminus \{0\}$ can be considered 
as a function $f\circ \Phi$ on $\mathcal{H}$, and with $r=e^{t/2}$ and $\theta$ such that $\tan\theta=\xi$, 
as $d\hat{\mu}(\xi,t)=e^{\delta t}dt\,d\nu_o(\xi)$, integration against $\hat{\mu}$ is given in polar coordinates on $\R^2\setminus\{0\}$  by
$$
\int_{\mathcal{H}} f\circ\Phi(\xi,t) d\hat{\mu} =
\int_{0}^{2\pi}\int_0^\infty f(r\sin\theta,r\cos\theta)2r^{2\delta-1}drd\bar{\nu}_o(\theta).
$$

When $S$ has finite volume, we have $\delta=1$, $d\bar{\nu}_o(\theta)=\frac{d\theta}{2\pi}$, 
and $\bar{\mu}=\frac{\mbox{Leb}}{\pi}$, where Leb is the usual Lebesgue measure on $\R^2$. 
Moreover, using the expressions of $m^{ps}$, $m$, $\mu_{H^-}$, in terms of $\nu_o$ and 
the Busemman cocycle, and as $\lambda_{H^-}$ corresponds to the parametrization of the horocyclic flow, 
we can see that $d\lambda_{H^-}(v)=ds=\pi d\mu_{H^-}(v)$, $m=\pi m^{ps}$, and $\mbox{Liouv}=\pi^2 m^{ps}$, 
where $\mbox{Liouv}$ is the measure corresponding to $\mbox{Leb}$ in this duality. 

\subsubsection*{The cocycle $c$} For all $g\in \SL(2,\R)$ and $\vu\in\R^2\setminus\{0\}$, 
$\Psi(g\vu)$ and $g\Psi(\vu)$ 
lie on the same strong stable horocycle $\Phi^{-1}(g\vu)=g\Phi^{-1}(\vu)$. 
Therefore, we can define a cocycle $c_\vu(g)$ by the implicit equation
 $$
n_{c_\vu(g)}=\Psi(g\vu)^{-1}g\Psi(\vu).
$$ 
  It satisfies
  \eq{propc1}{
  c_{\vu_0}(gn_s)=c_{\vu_0}(g)+s, \, \, \, c_{\vu_0}(ga_t)=e^{-t} c_{\vu_0}(g),
  }
and 

\eq{propc2}{
c_{\vu}(g)=c_{\vu_0}(g\Psi(\vu)).
}

\subsubsection*{Lift of a map from $\R^2\setminus\{0\}$ to $T^1\H=\PSL(2,\R)$}

Given a continuous, symmetric map with compact support $f:\R^2\setminus\{0\}\to \R$, we want to define 
continuous  maps $\tilde{f}:T^1\H\to \R$ and $\bar{f}: \Gamma\backslash\PSL(2,\R)\to\R$ in such a way 
that $\int_{T^1M}\bar{f} \,dm=\int_{\R^2}f\,d\bar{\mu}$. 

Fix a non-negative function $\phi$, vanishing outside $[-1,1]$, 
such that $\int_\R \phi(t) dt=1$. 

To a symmetric $f$, we associate the function on $G=T^1\H$
 $$
\tilde{f}(g)=f(g\vu_0)\phi(c_{\vu_0}(g)),
$$
 and the function on $\Gamma\backslash G=T^1S$,
 $$
\bar{f}(\Gamma g)=\sum_{\gamma \in \Gamma} \tilde{f}(\gamma g),
$$
 which is continuous and compactly supported.

We have
\begin{eqnarray*}
\int_{T^1S} \bar{f}(\Gamma g)\,dm(\Gamma g)&=&\int_{T^1\H}\tilde{f}( g)\,d\widetilde{m}(g)=
\int_{\mathcal{H}} f(g\vu_0)\int_{\R}\phi(c_{\vu_0(g)})ds d\bar{\mu}\\
&=&  \int_{\mathcal{H}} f\,d\hat{\mu}=\int_{\R^2} f\,d\bar{\mu}
\end{eqnarray*}
where we write $\tilde{m}$ for the $\Gamma$-invariant lift of the Burger-Roblin measure $m$.

Given a symmetric function $f:\R^2\setminus\{0\} \rightarrow \R$, of compact support, 
and $\vu \in \R^2\setminus\{0\}$, we define the following quantities:
$$
R^{(\vu)}(f)=\sup_{\vv \in \supp f} \vv\star\vu, \; \;
 r^{(\vu)}(f)=\inf_{\vv \in \supp f} \vv\star \vu, \; \; 
v^{(\vu)}(f)=\frac{R^{(u)}(f)}{r^{(u)}(f)},
$$ 
that can be geometrically interpreted as the radii of the smallest annulus 
containing the support of $f$, and their ratio, with respect to the level sets
 of the proper map $\vv\mapsto \vv \star \vu$. 
We will also need the following
$$
D(\vu,f)= \sup_{\vv \in \supp(f)} ||\Psi(\vv)\Psi(\vu)^{-1}||,
$$
which satisfies
$$
D(\vu,f)\leq c_1 \sup_{\vv \in \supp(f)} \max\left( \frac{|\vv|}{|\vu|},\frac{|\vu|}{|\vv|} \right),
$$
for some constant $c_1>0$.


\subsubsection*{Strategy of the proof, and key lemmas}

Heuristically, the proof of theorems about the distribution of nondiscrete 
$\Gamma$-orbits on $\R^2$ is based on the fact 
that it is possible to relate the sum $\sum_{\gamma\in\Gamma_T} f(\gamma\vu)$ 
to an integral of $\bar{f}$ along a certain piece of the horocycle 
$\Gamma\Phi^{-1}(\vu)$ on $\Gamma\backslash\PSL(2,\R)$, and then to
 use the equidistribution properties of the horocyclic flow on 
$\Gamma\backslash\PSL(2,\R)$ (theorem \ref{magic-equivalent}) to get the desired 
equivalent or the desired limit. 

A key observation (see lemma 2.1 of \cite{MW}, but also \cite{led}
 lemma 3 for the usual norm)
 is that if $f:\R^2\setminus\{0\}$ is continuous with compact support,
 $\vu\in\R^2\setminus\{0\}$, and
$g\in G$ is such that $g\vu $ belongs to the support of $f$, then 
$$
{\bigg |} \|g\| -|c_{\vu}(g)(g\vu\star\vu)| {\bigg |} \le D=D(\vu,f)\,.
$$

This allows (see  Lemma 3.1 of \cite{MW}) to get the following lemma. 

\begin{lem} \label{lemme31MW}
Let $\vu \in \R^2\setminus \{0\}$, $f$ a symmetric, nonnegative,
 continuous function compactly supported on $\R^2\setminus \{0\}$.
Then for all $\gamma \in\Gamma_0$,
\eq{minoration1}{ 
||\gamma||\leq T \Rightarrow \int_{-(1+(T+D)/r)}^{1+(T+D)/r} \tilde{f}(\gamma \Psi(\vu) n_s)ds = f(\gamma\vu),
}
\eq{majoration1}{ 
||\gamma||\geq T \Rightarrow \int_{-((T-D)/R-1)}^{(T-D)/R-1} \tilde{f}(\gamma \Psi(\vu) n_s)ds = 0,
}
where $D=D(\vu,f)$, $R=R^{(\vu)}(f)$ and $r=r^{(\vu)}(f)$.
\end{lem}

 This implies the following estimate of the sums over $\Gamma_T$ 
by ergodic averages along the horocycles.
 
 \begin{lem} \label{estimeelemme}
 With the same notations as in the Lemma \ref{lemme31MW},
 \eq{majoration2}{
 \sum_{\gamma \in \Gamma_T} f(\gamma \vu) \leq 2\int_{-(1+(T+D)/r)}^{1+(T+D)/r} \bar{f}(h^s u)ds,
 }
 \eq{minoration2}{
 \sum_{\gamma \in \Gamma_T} f(\gamma \vu) \geq 2\int_{-((T-D)/R-1)}^{(T-D)/R-1} \bar{f}(h^s u)ds.
 }
 \end{lem}
 \begin{proof}
  Summing the equality (\ref{eq: minoration1}) over $\Gamma_T$, we have
$$ 
\sum_{\gamma \in \Gamma_T} f(\gamma \vu) = 
\int_{-(1+(T+D)/r)}^{1+(T+D)/r} \sum_{\gamma \in \Gamma_T} \tilde{f}(\gamma\Psi(\vu)n_s)ds
 \leq 2\int_{-(1+(T+D)/r)}^{1+(T+D)/r} \bar{f}(h^s u)ds.
$$
 The factor $2$ in the previous equation comes from the fact that 
there are two elements $\pm\gamma$ of $\Gamma_0\subset\SL(2,\R)$ 
in the class of an element $\gamma \in \Gamma\subset \PSL(2,\R)$. 
 This proves (\ref{eq: majoration2}). For (\ref{eq: minoration2}), consider
 \begin{align*}
 \int_{-((T-D)/R-1)}^{(T-D)/R-1} \bar{f}(h^s u)ds
 &=\sum_{\gamma \in \Gamma} \int_{-((T-D)/R-1)}^{(T-D)/R-1} \tilde{f}(\gamma\Psi(\vu)n_s)ds\\
 &= \frac12 \sum_{\gamma \in \Gamma_T} \int_{-((T-D)/R-1)}^{(T-D)/R-1} \tilde{f}(\gamma\Psi(\vu)n_s)ds\\
 \end{align*}
 because of (\ref{eq: majoration1}), the factor $1/2$ coming from the map $\Gamma_0\rightarrow \Gamma$. 
Thus, (\ref{eq: minoration2}) is a consequence of the fact 
that for all $t>0$, $\int_{-t}^{t} \tilde{f}(\gamma \Psi(\vu) n_s)ds \leq f(\gamma \vu)$.
 \end{proof} 
 
  We now take into account the scaling by $T^\alpha$, 
and use the equidistribution Theorem \ref{magic-equivalent}.
  Let $\alpha \in (-1,1)$, and $\eta>0$ be a small parameter. For $f$ a continuous map with compact support on $\R^2\setminus\{0\}$, 
introduce the following quantity\,:
$$
I(\alpha,f,T,\vu)=\frac{\sum_{\gamma \in \Gamma_T} f\left(\frac{\gamma \vu}{T^\alpha}\right)}{T^{(1+\alpha)\delta}},
$$

 \begin{lem} 
 Take the same notations as in Lemma \ref{lemme31MW}. If $\Gamma$ is convex-cocompact and $\alpha\in (-1,1)$, or if $\Gamma$ is geometrically finite 
and $\alpha=0$, then the quantity
 $ 
I(\alpha,f,T,\vu)
$
satisfies 
\eq{minoration3}{
\liminf_{T\rightarrow +\infty} \frac{I(\alpha,f,T,\vu)}{\tau(g^{(1-\alpha)\log T - \log R- \eta}\Gamma\Psi(u))} \geq
 \frac{2e^{-\eta}}{v^{(\vu)}(f)m^{ps}(T^1S)} \int_{\R^2} \frac{f(\vv)}{(\vv\star \vu)^\delta}d\bar{\mu}(\vv)  ,
 }
and 
\eq{majoration3}{
 \limsup_{T\rightarrow +\infty} \frac{I(\alpha,f,T,\vu)}{\tau(g^{(1-\alpha)\log T - \log r +\eta}\Gamma\Psi(u))} 
 \leq \frac{2e^{\eta} v^{(\vu)}(f)}{m^{ps}(T^1S)} \int_{\R^2} \frac{f(\vv)}{(\vv\star \vu)^\delta}d\bar{\mu}(\vv).  
}
 \end{lem}

In the geometrically finite case, the lack of compacity of $\Omega$ does not allow to get a uniform convergence in theorem \ref{magic-equivalent}, 
and therefore we are not able to rescale our estimates through the parameter $\alpha$ (see the proof below for details). 

 \begin{proof}
  Put $\vu_T=\vu/T^\alpha$. 
Then $\Psi(\vu_T)=\Psi(\vu)a_{-2\alpha\log T}$, $R^{(\vu_T)}(f)=R/T^\alpha$, $r^{(\vu_T)}(f)=r/T^\alpha$, and 
  as $T$ goes to infinity, $D(\vu_T,f)=O(T^{|\alpha|})$. Thus, for $T$ sufficiently large,
$$
1+\frac{T+D(\vu_T,f)}{r^{(\vu_T)}(f)}\leq e^\eta \frac{T^{1+\alpha}}{r},
$$
  and 
  $$
 \frac{T-D(\vu_T,f)}{R^{(\vu_T)}(f)}-1 \geq e^{-\eta}\frac{T^{1+\alpha}}{R}.
$$
 Now apply inequality (\ref{eq: minoration2}) 
to the function $f$ and the vector $\vu_T$. We obtain
 $$
\sum_{\gamma \in \Gamma_T} f\left(\frac{\gamma \vu}{T^\alpha}\right)
 \geq 2\int_{-e^{-\eta}T^{(1+\alpha)}/R}^{e^{-\eta}T^{(1+\alpha)}/R} \bar{f}(h^s g^{-2\alpha\log T}\Gamma\Psi(\vu))ds.
 $$
When $\alpha=0$, $g^{-2\alpha\log T}\Gamma\Psi(\vu)$ is constant. 
Thus, using  Theorem \ref{magic-equivalent}, when $\alpha=0$, we obtain 
$$
\liminf_{T\rightarrow +\infty} \frac{I(0,f,T,\vu)}{\tau(g^{\log T - \log R- \eta}\Gamma\Psi(\vu))}
 \geq \frac{2e^{-\delta \eta} R^{-\delta}}{m^{ps}(T^1S)}\int_{T^1 S} \bar{f}dm
 \geq \frac{2e^{-\delta\eta} R^{-\delta} }{m^{ps}(T^1S)}\int_{\R^2} fd\bar{\mu},
$$
 and since for all $\vv \in \supp(f)$, $\vv\star \vu\geq r=v^{(\vu)}(f)R$, 
inequality (\ref{eq: minoration3}) follows. The upper bound 
 (\ref{eq: majoration3}) is similar.

Otherwise (when $\alpha\neq 0$), we need a uniform equidistribution property.
When $\Gamma$ is convex-cocompact,  Theorem \ref{magic-equivalent} gives 
a uniform convergence of 
$$
\frac{1}{e^{-\delta\eta}T^{(1+\alpha)\delta}/R^\delta }\frac{1}{\tau(g^{(1+\alpha)\log T-\log R-\eta}w)}
\int_{-e^{-\eta}T^{(1+\alpha)}/R}^{e^{-\eta} T^{(1+\alpha)}/R} \bar{f}(h^s g^{-2\alpha\log T} w)\,ds
$$
towards $\int_{T^1M}f\,dm/m^{ps}(T^1M)$, uniformly in $w\in \Omega$. 
However, we can not apply it directly, since there is no reason that $g^{-2\alpha\log T}\Gamma\Psi(\vu)$ belongs to $\Omega$. 
Choose $s_0\in \R$ such that $\Gamma\Psi(\vu)=h^{s_0}u_0$, with $u_0\in \Omega$. 
Now, $g^{-2\alpha\log T }u_0\in \Omega$, so that the above convergence holds when replacing $w$ by $g^{-2\alpha\log T}u_0$.  

Observe that $h^s g^{-2\alpha\log T} h^{s_0}u_0=h^{s+s_0T^{2\alpha}} g^{-2\alpha\log T}u_0$, so that 
\begin{eqnarray*}
\int_{-e^{-\eta}T^{(1+\alpha)}/R}^{ e^{-\eta}T^{(1+\alpha)}/R}\bar{f}(h^s\Gamma\Psi(\vu)&=& 
\int_{-e^{-\eta}T^{(1+\alpha)}/R+ s_0T^{2\alpha}}^{ e^{-\eta}T^{(1+\alpha)}/R+ s_0T^{2\alpha}}\bar{f}(h^s g^{-2\alpha\log T}u_0)\,ds\\
&\ge& \int_{-e^{-\eta}T^{(1+\alpha)}/R+ s_0T^{2\alpha}}^{ e^{-\eta}T^{(1+\alpha)}/R- s_0T^{2\alpha}}\bar{f}(h^s g^{-2\alpha\log T}u_0)\,ds -2s_0 T^{2\alpha}\|f\|_\infty
\end{eqnarray*}
As $\alpha\in (-1,1)$, we have 
$$
\lim_{T\to +\infty} \frac{s_0 T^{2\alpha}\|f\|_\infty}{T^{(1+\alpha)\delta}\tau(g^{(1+\alpha)\log T - \log R- \eta} h^{s_0} u_0)} = 0 
$$
Remark also that $\tau$ is continuous, and therefore uniformly continuous on a compact neighbourhood $V(\Omega)$ of $\Omega$. 
As $(1-\alpha)\log T-\log R-\eta\to +\infty$ when $T\to +\infty$, the distance from $g^{(1-\alpha)\log T-\log R-\eta}\Gamma\Psi(\vu)$ to 
$g^{(1-\alpha)\log T-\log R-\eta}u_0\in \Omega$ goes to $0$. Thus, when $T$ goes to $+\infty$, the ratio 
$$
\frac{\tau(g^{(1-\alpha)\log T - \log R- \eta}\Gamma\Psi(\vu))}{\tau(g^{(1-\alpha)\log T-\log R-\eta}u_0)}$$
converges to $1$. 

The uniform convergence property of theorem \ref{magic-equivalent} on $\Omega$ gives, for 
$T$ large enough, and all $u_0\in \Omega$,  
$$ 
\frac{\int_{-e^{-\eta}T^{(1+\alpha)}/R+ s_0T^{2\alpha}}^{ e^{-\eta}T^{(1+\alpha)}/R- s_0T^{2\alpha}}\bar{f}(h^s g^{-2\alpha\log T}u_0)\,ds}
{e^{-\eta}T^{(1+\alpha)}/R\tau(g^{(1-\alpha)\log T-\log R-\eta}u_0)}\ge 
e^{-\eta} \frac{\int_{T^1M}\bar{f}\,dm}{m^{ps}(T^1M)}=e^{-\eta} \frac{\int_{\R^2} f\,d\bar{\mu}}{m^{ps}(T^1M)}
$$

Now, putting all these estimates together, we obtain, for $T$ large enough, 
$$
 \liminf_{T\to +\infty}\frac{I(\alpha,f,T,\vu)}{\tau(g^{(1-\alpha)\log T - \log R- \eta}\Gamma\Psi(\vu))} 
 \geq \frac{2e^{-\delta\eta-3\eta} R^{-\delta} }{m^{ps}(T^1S)}\int_{\R^2} fd\bar{\mu},
$$
 and since for all $\vv \in \supp(f)$, $\vv\star \vu\geq r=v^{(\vu)}(f)R$, 
inequality (\ref{eq: minoration3}) follows. The upper bound 
 (\ref{eq: majoration3}) is similar.
 \end{proof}

 
\subsection{Proof of Theorem \ref{th:convexcocompact} }

 It is sufficient to consider nonnegative $f$.
 As $\Gamma_T=-\Gamma_T$, we have 
 $\sum_{\gamma \in \Gamma_T} f(\gamma \vu/T^\alpha) = 
\sum_{\gamma \in \Gamma_T} \frac12(f(\gamma \vu/T^\alpha)+f(-\gamma \vu/T^\alpha))$, so that
 we can also restrict the study to symmetric $f$.
 
Let $f,\vu$ be as in the assumptions of Theorem \ref{th:convexcocompact}, 
with $f$ nonnegative and symmetric. 
Let $u\in T^1S$ be the projection of $\Psi(\vu)$, and $\alpha \in (-1,1)$. 
Fix  $\epsilon=\eta=1/2$, and decompose $f$ as
 $$
f=\sum_i f_i,
$$
 where $v^{(\vu)}(f_i)<e^\varepsilon$. 

Since $\Gamma$ is convex-cocompact, the nonwandering set $\Omega$ is compact, 
and the continuous map $\tau$, which is positive on $\Omega$,  is bounded from below and above by positive constants $c_1,c_2$ on the set 
$\Omega_{1/2}$ of vectors $v\in \mathcal{E}$, such that $h^s v\in\Omega$ for some $s$ with $|s|\le \frac{1}{2}$. 
As $\vu\in \mathcal{C}(\Gamma_0)$, the horocycle $\Gamma \Phi(\vu)$ is included in  $\mathcal{E}$,
 so that for all $t>0$ large enough, $g^t\Gamma\Psi(\vu)$ belongs to $\Omega_{1/2}$, and  
 $$
c_1 \leq \tau(g^t u) \leq c_2.
$$
 Inequalities (\ref{eq: minoration3}) and (\ref{eq: majoration3}) for the map $f_i$ imply that
 $$\frac{\sum_{\gamma \in \Gamma_T} f_i\left(\frac{\gamma \vu}{T^\alpha}\right)}{T^{(1+\alpha)\delta}}\asymp \int_{\R^2} \frac{f_i(\vv)}{(\vv\star \vu)^\delta}d\bar{\mu}(\vv),$$
 where the implied constants do not depend on $f_i$ nor $\vu$. 
Summing over $i$ gives the required estimate.

\begin{rem}\rm Observe now, to prove the assertion of remark \ref{boundedorbits},
 that if $\Gamma$ is geometrically finite, but the geodesic orbit $(g^t\Gamma\Psi(\vu))_{t\ge 0}$ 
is bounded when $t\to +\infty$, there exist two constants $0<c_1(\vu)\le c_2(\vu)<\infty$, such that 
$c_1(\vu) \leq \tau(g^t u) \leq c_2(\vu)$. 
Thus, the conclusion of theorem \ref{th:convexcocompact} remains valid for vectors $\vu\in\R^2\setminus\{0\}$ such that
the positive geodesic orbit of $\Gamma\Psi(\vu))$ is bounded. 
\end{rem}


\subsection{Proof of Proposition \ref{theresnolimit}}

 We begin by proving two Lemmas.  

\begin{lem} Assume that $\Gamma$ is a nonelementary geometrically finite group of infinite covolume. 
The map $\tau$ restricted to $\Omega$ is non-constant on any orbit of the geodesic flow.
\end{lem}

\begin{proof} Since $\Gamma$ is geometrically finite of infinite volume, 
the set of ordinary points $\partial \H \setminus \Lambda$ is an open dense set.
 Let $u \in T^1S$ , $v \in T^1\H$ a lift. By definition
\begin{align*}
\tau(g^t u) & =\mu_{H^-}((h^sg^tu)_{|s|\leq 1})=e^{-\delta t}\mu_{H^-}((h^s u)_{|s|\leq e^t}\\
& = e^{-\delta t} \int_{-e^t}^{e^t} e^{\delta \beta_{(h_s v)^-}(o,\pi(h_s v))}d\nu_o((h_s v)^-)
\end{align*}
 The set of $s\in \R$ such that $(h_sv)^-$ is the ordinary set 
$\partial \H \setminus \Lambda$ is an open dense subset. 
So the set of $t\geq 0$ such that both $(h_{e^t}v)^-$ and $(h_{-e^t}v)^-$ 
are both in the ordinary set is also an open dense subset of $\R^+$.
 In particular, the above integral is locally constant on an
 open dense set of parameters $t\in\R$; 
since the factor $e^{-\delta t}$ is strictly decreasing, this proves the claim.
\end{proof}

\begin{lem} \label{variation} Assume that $\Gamma$ is a nonelementary geometrically finite 
 group of infinite covolume. 
For all $t_0>0$, there exists $\epsilon>0$ such that for $\hat{\mu}$-almost every horocycle of $\mathcal{H}$
and every $u$ in such an horocycle,
$$
\limsup_{t\rightarrow +\infty} \frac{\tau(g^{t+t_0}u)}{\tau(g^{t}u)}>e^\epsilon,
$$
and
$$
\liminf_{t\rightarrow +\infty} \frac{\tau(g^{t+t_0}u)}{\tau(g^{t}u)}<e^{-\epsilon}.
$$
Moreover,  $\epsilon$ can be choosen uniformly in $t_0$ varying in compact subsets of $\R^+_*$.
\end{lem}

\begin{proof} Define 
$\displaystyle 
c^+=\sup_{u\in \Omega} \frac{\tau(g^{t_0}u)}{\tau(u)}$  and 
$\displaystyle c^-=\inf_{u\in \Omega} \frac{\tau(g^{t_0}u)}{\tau(u)}.$
We will first prove that $c^-<1<c^+$, and then deduce the lemma. 

Assume that $c^+\leq 1$, and $\Gamma$ is convex-cocompact. 
Then for all $u \in \Omega$, $\tau(g^{t_0}u)\leq \tau(u)$,
 so the Birkhoff averages
 satisfy $\frac1n \sum_{k=0}^{n-1} \tau(g^{kt_0}u)\leq \tau(u)$. 
Since the geodesic flow is mixing with respect to $m^{ps}$, 
 $g^{t_0}$ is ergodic with respect to $m^{ps}$. 
As $\tau$ is continuous on the compact $\Omega$,
 it is integrable w.r.t. $m^{ps}$. 
The Birkhoff Theorem implies that for $m^{ps}$-a.e. $u\in\Omega$, 
$\frac{1}{m^{ps}(\Omega)}\int_\Omega \tau dm^{ps}\leq \tau(u)$.
As $\tau$ is continuous,  this is valid for all $u\in \Omega$, 
so that $\tau$ is constant on $\Omega$, which is a contradiction. 
So $c^+>1$, and similarly, $c^-<1$.

Let us deal now with the case where  $\Gamma$ is geometrically finite. 
We use results proved in section \ref{geometrically-finite-groups}.
By proposition \ref{mesure-des-boules-horospheriques}, there exists $T_0>0$, such that for 
all $t_0\ge T_0$, there exists a vector $u$ on  a unbounded geodesic, satisfying 
$\frac{\tau(g^{t_0}u)}{\tau(u)}>1$. Now, if $0<t_0<T_0$, $nt_0\ge T_0$ for some $n\ge 1$, so that 
there exists $v\in T^1M$, with $\tau(g^{nt_0}v)>\tau(v)$. But 
$\displaystyle \frac{\tau(g^{nt_0}v)}{\tau(v)}=\Pi_{k=1}^n\frac{\tau(g^{kt_0}v)}{\tau(g^{(k-1)t_0}v)}$, 
so that we deduce also that $c^+>1$. 

 Choose a small $\eta>0$, and Let $v \in \Omega$ such that $\frac{\tau(g^{t_0}v)}{\tau(v)}\geq c^+-\eta$. 
An ergodicity argument (still valid on geometrically finite surfaces) shows 
that  for $m^{ps}$-almost every $u$ there is a sequence $t_i\rightarrow +\infty$ 
such that $g^{t_i}u\rightarrow v$, so that 
$$
\limsup_{t\rightarrow +\infty} \frac{\tau(g^{t+t_0}u)}{\tau(g^{t}u)}\geq c^+-\eta\,.
$$
Letting $\eta>0$ go to $0$, the previous limsup is in fact $\geq c^+$.

When $\Gamma$ is convex-cocompact, the map $\tau$ is continuous
 with compact support in a bounded neighbourhood of $\Omega$, and therefore 
uniformly continuous. 
It implies that the above equality depends only on the stable horocycle of $u$.
When $\Gamma$ is geometrically finite,  it is also true, thanks to  lemma \ref{indephoro}. 
 Since $\hat{\mu}$ is the transversal measure on $\mathcal{H}$ of $m^{ps}$,
 this equality is true $\hat{\mu}$-almost surely.
The $\liminf$ case is similar. 
The uniformity in $t_0$ on compact sets 
follows from the continuity of $c^+$ and $c^-$ as functions of $t_0$.
\end{proof}

Let us deduce now Proposition \ref{theresnolimit} from this last Lemma.

 Let $\epsilon>0$ be given by Lemma \ref{variation} for all $t_0 \in [\log (5/4),\log 4]$.
 Choose $\vv_0 \in \mathcal{C}=\supp(\bar{\mu})$.\

Let $f$ be a bump function in a small neigbourhood of $\vv_0$. 
Let $g(\vv)=f(2\vv)$, $\eta>0$ and $\vu \in \R^2-\{0\}$.
The choice of $f$ can be done in such a way that for all $\vu$, 
$$
3 \geq \frac{R^{(\vu)}(g)}{r^{(\vu)}(f)} \geq \frac{r^{(\vu)}(g)}{R^{(\vu)}(f)}\geq \frac32,
$$
and $v^{(\vu)}(f)=v^{(\vu)}(g)<e^{\epsilon/4}$. 
 Using (\ref{eq: minoration3}) and (\ref{eq: majoration3}) with $f$ and $g$ respectively 
and $\alpha=0$, we have for $T$ sufficiently large
$$
\frac{\sum_{\gamma \in \Gamma_T} f(\gamma\vu)}{\sum_{\gamma \in \Gamma_T} g(\gamma\vu)}\geq
 \frac{e^{-3\eta}}{v^{(\vu)}(f)v^{(\vu)}(g)}\frac{\tau(g^{\log T-\log R^{(u)}(f)-\eta}u)}
 {\tau(g^{\log T-\log r^{(u)}(g)+\eta}u)}
 \frac{\int \frac{f(\vv)}{(\vv\star \vu)^\delta}d\bar{\mu}(\vv)}{\int \frac{g(\vv)}{(\vv\star \vu)^\delta}d\bar{\mu}(\vv)}.$$
 Similarly,
$$
\frac{\sum_{\gamma \in \Gamma_T} f(\gamma\vu)}{\sum_{\gamma \in \Gamma_T} g(\gamma\vu)}\leq 
e^{3\eta}v^{(\vu)}(f)v^{(\vu)}(g) 
\frac{\tau(g^{\log T-\log r^{(u)}(f)+\eta}u)}{\tau(g^{\log T-\log R^{(u)}(g)-\eta}u)} 
\frac{\int \frac{f(\vv)}{(\vv\star \vu)^\delta}d\bar{\mu}(\vv)}{\int \frac{g(\vv)}{(\vv\star \vu)^\delta}d\bar{\mu}(\vv)}.
$$
Define $t_1=-\log R^{(u)}(f)+\log r^{(u)}(g)-2\eta$ and $t_2=-\log r^{(u)}(f)+\log R^{(u)}(g)+2\eta$.
 For small enough $\eta$, $t_1$ and $t_2$ are both in the compact subset $[\log (5/4),\log 4]$. 
Assume also that $e^\eta<e^{\epsilon/12}$. 
Lemma \ref{variation} applied to $t_1$ gives
 $$
\limsup_{T\rightarrow +\infty} \frac{\sum_{\gamma \in \Gamma_T} f(\gamma\vu)}{\sum_{\gamma \in \Gamma_T} g(\gamma\vu)} \geq
 e^{-3\epsilon/4}e^{\epsilon}
 \frac{\int \frac{f(\vv)}{(\vv\star \vu)^\delta}d\bar{\mu}(\vv)}{\int \frac{g(\vv)}{(\vv\star \vu)^\delta}d\bar{\mu}(\vv)}> 
 \frac{\int \frac{f(\vv)}{(\vv\star \vu)^\delta}d\bar{\mu}(\vv)}{\int \frac{g(\vv)}{(\vv\star \vu)^\delta}d\bar{\mu}
(\vv)},$$
 and similarly for the $\liminf$,
$$
\liminf_{T\rightarrow +\infty} \frac{\sum_{\gamma \in \Gamma_T} f(\gamma\vu)}{\sum_{\gamma \in \Gamma_T} g(\gamma\vu)} 
< \frac{\int \frac{f(\vv)}{(\vv\star \vu)^\delta}d\bar{\mu}(\vv)}
{\int \frac{g(\vv)}{(\vv\star \vu)^\delta}d\bar{\mu}(\vv)}.
$$
This concludes the proof of Proposition \ref{theresnolimit}.


\section{Geometrically finite groups}\label{geometrically-finite-groups}

In this section, $S$ is a nonelementary geometrically finite surface with cusps. 
It can be written as the union of a compact set, a finite union of cusps, 
and a finite union of funnels. From a dynamical point of view, 
the nonwandering set $\Omega$  of 
the geodesic flow does not see the funnels, and is therefore
 the union of a compact set $K_0$, and finitely many cusps $C_1,...,C_k$. 

If $u\in \Omega$, then $u^+$ is either radial or parabolic. 
If $u^+$ is radial, we  say that $u$   is radial; 
it means that $g^tu$ returns infinitely often in a compact set (which depends on $u^+$)
when $t\to +\infty$. 
On a geometrically finite surface, enlarging $K_0$ a little,
 we may assume that every radial vector returns to $K_0$ infinitely often.


\subsection{The Shadow Lemma}

All results and methods of this paragraph use directly those of \cite{Scha1}. 
The reader should note that in this article,
 the author considered the strong unstable horocyclic 
flow, whereas we use here the strong stable horocyclic flow. 
In particular, the results might at first glance look different. 

The notation $a=C^{\pm 1} b$ means that $ C^{-1}b\le a\le C b$. 
In this paragraph, we prove the following Proposition. 

\begin{prop}\label{mesure-des-boules-horospheriques} 
Let $S$ be a hyperbolic nonelementary  geometrically finite surface,
 $\Omega\subset T^1S$ the nonwandering set of the geodesic flow, 
and $K_0\subset \Omega$ the compact part of $\Omega$. 

Then there exists a constant $C\ge 1$, 
such that for $u\in \Omega$, we have
$$
\tau(u)=C^{\pm1} e^{(1-\delta)d(\pi(u),K_0)}.
$$
\end{prop}

\begin{rem}\rm This result, which can be read as: for any $u \in \Omega$,
$$
\mu_{H^-}((h^su)_{|s|\leq r})=C^{\pm 1}r^\delta e^{(1-\delta)d(\pi(g^{\log r}u),K_0)},
$$
 is also true on manifolds of 
higher dimension and variable curvature, under the assumption $(*)$
 of theorem 3.2 of \cite{Scha1}. 
But the statement has to be slightly modified. 
Assume that the surface $S$ has $k$ cusps $C_i$, 
$1\le i\le k$, whose critical exponents are 
denoted by $\delta_i<\delta$. 
Then we have 
$$
\mu_{H^-}((h^su)_{|s|\le r})=
C^{\pm 1}r^{\delta} \, e^{(2\delta_i-\delta)d(\pi(g^{\log r}u),K_0)}\,,
$$ if $g^{\log r}u$ belongs to the cusp $C_i$. 
\end{rem}

 We will deduce  Proposition \ref{mesure-des-boules-horospheriques} 
from the following version of the {\em Shadow Lemma}.
 For $x\in\H$, $\xi\in S^1$, $t\in \R$, denote by $\xi_x(t)$ 
the point at signed distance $t$ of $x$ on the geodesic $(x\xi)$, 
with the orientation such that $\xi_x(t)\rightarrow \xi$ as
 $t\rightarrow +\infty$. Define $V(x,\xi,t)$ 
as the set of points $\eta\in \partial \H$ whose projection on the geodesic 
line $(x\xi)$ is in fact on the geodesic ray $[\xi_x(t)\xi)$.

\begin{thm}[\cite{Scha1},Th 3.2]\label{shadow-lemma}
 Let $S$ be a nonelementary geometrically finite hyperbolic surface.
Fix a point $x \in \H$. 
 There exists a constant $A>0$ such that 
for all $\xi\in \Lambda$ and $t\ge 0$, we have
$$
\nu_x(V(x,\xi,t))= A^{\pm 1} e^{-\delta t+(1-\delta)d(\xi_x(t),\Gamma.x)}\,. 
$$
Moreover, the constant $A$ can be chosen in a $\Gamma$-invariant way, and 
uniformly in $x$ varying in a compact set. 
\end{thm}

For $\tilde{u}\in T^1\H$, denote by $Pr_{\tilde{u}^+}:(h^s\tilde{u})_{s\in\R}\to\partial \H\setminus\{\tilde{u}^+\}$ 
the natural projection which sends $v$ to $v^-$. 
Small pieces of orbits of $(h^s)$ are almost sent 
to sets of the form $V(x,\xi,t)$, according to the 
following lemma (see for example \cite[lemma 4.4 page 982]{Scha1} for a proof).

\begin{lem}\label{projection} 
There exists a constant $\alpha\geq 0$, 
such that for all 
$\tilde{v}\in T^1\H$, and $t\ge \alpha$, we have 
$$
V(x,\tilde{v}^-,t+\alpha)\subset Pr_{\tilde{v}^+}((h^s\tilde{v})_{|s|\le e^{-t}})
\subset V(x,\tilde{v}^-,t-\alpha)\,,
$$
where $x=\pi(\tilde{v})\in\H$ is the base point of $\tilde{v}$. 
\end{lem}

\begin{proof}[Proof of Proposition \ref{mesure-des-boules-horospheriques}]
First note that it is sufficient to prove the Proposition for a vector $u$ which is radial. 
Indeed, $\tau$ is continuous, the set of radial vectors is dense in $\Omega$, 
and $C$ is an absolute constant valid for all radial vectors. 
 
Let $u\in \Omega$. Let $t>\alpha$ be the first time such that $g^tu\in K_0$. 
Choose  the constant $A$ in theorem \ref{shadow-lemma}   so that the theorem 
is valid with this same constant $A$ for all points $y\in\Gamma K_0$.  
Let $x=\pi(g^t\tilde{u})$ be the base point of $g^t\tilde{u}$. 

We have $\tau(u)=e^{\delta t}\mu_{H^-}((h^sg^tu)_{|s|\le e^{-t}})$. 
Moreover, 
$V(x,\tilde{u}^-,t+\alpha)\subset  Pr_{\tilde{u}^+}((h^s\tilde{u})_{|s|\le e^{-t}})
\subset V(x,\tilde{u}^-,t-\alpha)$. 
And for $\xi\in V(x,\tilde{u}^-,0)$, we have $|\beta_{\xi}(x,\pi(h^sg^Tu))|\le 1$.
Thus, using the expression 
$d\mu_{H^-(g^tu)}(h^sg^tu)=
e^{\delta\beta_{(h^sg^t\tilde{u})^-}(x,\pi(h^sg^T\tilde{u})}d\nu_x((h^sg^t\tilde{u}^-)$, lemma \ref{projection} and theorem \ref{shadow-lemma}, we get 
\begin{eqnarray*}
\mu_{H^-}((h^sg^tu)_{|s|\le e^{-t}})&\le& 
\mu_{H^-}(Pr_{\tilde{u}^+}^{-1}(V(x,\tilde{u}^-,t-\alpha))\\
&\le& C\nu_{x}(V(x,\tilde{u}^-,t-\alpha))
\le D e^{-\delta t}e^{(1-\delta)d(\pi(u),K_0)}\,.
\end{eqnarray*}
Similarly, we get $\mu_{H^-}((h^sg^tu)_{|s|\le e^{-t}})\ge D^{-1}e^{-\delta t}e^{(1-\delta)d(\pi(u),K_0)}$, so 
that $\tau(u)=D^{\pm 1} e^{-\delta t}e^{(1-\delta)d(\pi(u),K_0)}$. 
\end{proof}




\subsection{Integrability of $\tau$}

We now discuss the proof of Theorem \ref{integrability-of-balls}.

We will follow closely the method of \cite{DOP} 
where they give criteria of finiteness of $m^{ps}$ 
in terms of convergence of Poincar\'e series. 
The constants that appear in the proof differ 
from one step to another, but are often denoted by $C$. 

 Recall that, as $S$ is geometrically finite, it is the union of 
 finitely many cusps $C_i$, $1\le i\le k$, a compact part $K_0$, and finitely many funnels. 
The Patterson-Sullivan measure $m^{ps}$ has its support in $\Omega$, and 
$\Omega=\Omega\cap T^1K_0\sqcup (\cup_{i=1}^k\Omega\cap T^1C_i)$. 
Since $m^{ps}$ is finite on compact sets,  
we just need to study the integrability of 
the map $\tau$ in a fixed cusp $C_1$. 

By proposition \ref{mesure-des-boules-horospheriques},
we know that this function is, up to multiplicative constants, 
equal to $\displaystyle f(u)=\exp\left((1-\delta)d(\pi(u),K_0)\right)$. 
It is sufficient to check the integrability of $f$ on $T^1C_1$. 

We lift $C_1$ to a horoball $Hor(\xi_1,t_1)=H$. 
The stabilizer $\Pi_1=Stab_{\Gamma}(\xi_1)$ 
acts cocompactly on $\left(\partial H\right)\setminus\{\xi_1\}$, 
and on $\Lambda\setminus\{\xi_1\}$.
 By choosing one of the two generators of $\Pi_1$, 
we will consider elements of $\Pi_1$ as elements of $\Z$. 
Let $I_0$ be a connected relatively compact 
fundamental domain for the action of $\Pi_1$ on $H$, 
and $J_0\subset\Lambda\setminus\{\xi_1\}$ 
its image under the natural projection 
$Pr_{\xi_1}:\partial H\setminus\{\xi_1\}\to \Lambda\setminus\{\xi_1\}$. 
This set $J_0$ is the fundamental domain 
for the action of $\Pi_1$ on $\Lambda\setminus\{\xi_1\}$. 
Without loss of generality, we can assume that $o$ belongs to $I_0$. 

Up to a set of $m^{ps}$-measure zero, we can lift a vector $v\in T^1C_1$ to a vector 
$\tilde{v}\in T^1 H$ 
in such a way that $v^-\in J_0$, and $v^+\in p.J_0$, for 
some $p\in \Pi_1$. 
Define $0\le t(v)\le T(v)$ by the fact that $g^{-t(v)}\tilde{v}$ and $g^{T(v)-t(v)}\tilde{v}$ belong
to $T^1\partial H$. In other words, $t(v)$ is the 
length of $(v^-v^+)$ between $I_0$ and $\tilde{v}$, 
and $T(v)$ is the total time spent by $(v^-v^+)$ in $H$. 

Define $q_1,q_2\in \Pi_1$ by the fact that $\pi(g^{-t(v)}\tilde{v})\in q_1.I_0$ and 
$\pi(g^{T(v)-t(v)}\tilde{v})\in q_2.I_0$. As all hyperbolic triangles are thin, 
by considering the triangle $(v^-,v^+,\xi_1)$, 
it is easy to see that $d(o,q_1.o)$ and $d(p.o,q_2.o)$ are bounded by a constant 
depending only on the diameter of hyperbolic triangles and on the diameter of $K_0$.

Remark that $d(\pi(v),K_0)=d(\pi(\tilde{v}),\partial H)$. 

\begin{lem}\label{distance-au-bord-horosphere} 
There exists a constant $C>0$,
 such that if $v\in T^1C_1$ is lifted to 
$\tilde{v}\in T^1 H$ as above, and $g^{-t(v)}\tilde{v}\in T^1 \partial H$, 
then for all $0\le t\le T(v)$, 
$$
d(\pi(g^tg^{-t(v)}\tilde{v}),\partial H)=\min(t,T(v)-t)\pm C \,.
$$
Moreover, $T(v)$  depends only on $v^-$ and $v^+$, and 
we have $\displaystyle T(v)=d(o,p.o)\pm C\,.$
\end{lem}

\begin{proof} Consider the hyperbolic triangle $(v^-,v^+,\xi_1)$. 
As $\H$ is a hyperbolic metric space in the sense of Gromov, 
all triangles are thin. 
In particular, there is a universal constant $D>0$, such that 
the distance between $\pi(g^{t-t(v)}\tilde{v})$ and 
one of the other sides $(v^-\xi_1)$ or $(v^+\xi_1)$ is less than $D$. 
Let $x_t$ be the projection of $\pi(g^{t-t(v)}\tilde{v})$ on this closest side. 
Assume first 
that $x_t\in (v^-\xi_1)$. 
Let $x$ be the intersection of $(v^-\xi_1)$ with $\partial H$. 
By definition of a horoball centered in $\xi_1$, 
we have $d(x_t,\partial H)=d(x_t,x)$. 
As $(v^-\xi_1)$ and $(v^-v^+)$ are negatively asymptotic, 
we have $d(\pi(g^{-t(v)}\tilde{v}),x)\le d(\pi(g^{t-t(v)}\tilde{v}),x_t)\le D$.  
It implies that $|d(x_t,x)-d(\pi(g^{t-t(v)}\tilde{v}),\pi(g^{-t(v)}\tilde{v}))|\le 2D$. 

In particular, we deduce that 
$$
d(\pi(g^{t-t(v)}\tilde{v}),\partial H)=d(x_t,\partial H)\pm D=t\pm 3D \,.
$$

If $x_t$ belongs to $(v^+\xi_1)$, 
the same reasoning with the other intersection $\pi(g^{T(v)-t(v)}\tilde{v})$ of $(v^-v^+)$ 
with $\partial H$ will imply that 
$$
d(\pi(g^{t-t(v)}\tilde{v},\partial H)=\min(t,T(v)-t)\pm 3D,.$$

It remains to compare $T(v)$ with $d(o,p.o)$. 
By definition of $q_1$ and $q_2$, 
we know that $\pi(g^{-t(v)}\tilde{v})\in q_1.I_0$ and $\pi(g^{T(v)-t(v)}\tilde{v})\in q_2.I_0$, so that 
$T(v)=d(q_1.o,q_2.o)\pm 2diam(I_0)=d(o,p.o)\pm 2diam(I_0)\pm d(p.o,q_2.o) \pm d(o,q_1.o)$. 
Since $d(o,q_1.o)$ and $d(p.o,q_2.o)$ are bounded, the lemma is proved. 
\end{proof}

Let us continue now the proof of the theorem. 
For any function $f:T^1C_1\to \R$, lifted in $\tilde{f}:T^1H\to\R$, we have
\begin{align}\label{decomposition-de-la-mesure}
\int_{T^1C_1}f\,dm^{ps} = \sum_{p\in\Pi_1} 
\int_{J_0\times p.J_0} e^{\delta(\beta_{v^+}(o,\pi(v))+\beta_{v^-}(o,\pi(v))) }\nonumber\\
 \left(\int_{[0,T(v)]}\tilde{f}(g^tg^{-t(v)}\tilde{v})dt\right) d\nu_o(v^+)d\nu_0(v^-)  
\end{align}

The triangular inequality implies that 
$|\beta_{v^+}(o,\pi(v))+\beta_{v^-}(o,\pi(v))|\le 2d(o,(v^-v^+))\le 2 d(o,q_1.o)+diam(I_0)$, which is bounded. 
Thus, up to some multiplicative constants, for any nonnegative continuous function $f$ on $T^1C_1$, 
the integral $\int_{T^1C_1}f dm^{ps}$ is equal to 
$$
\sum_{p\in \Pi_1}\int_{J_0}d\nu_0(v^-)\int_{p.J_0}d\nu_o(v^+)\int_{[0,T(v)]}\tilde{f}(g^tg^{-t(v)}\tilde{v})dt\,.
$$
Apply this in the particular case where $f(v)=e^{(1-\delta)d(\pi(v),K_0)}$. 

By lemma \ref{distance-au-bord-horosphere}, there exists a constant such that 
$\tilde{f}(\tilde{v})=C^{\pm 1} e^{(1-\delta)\min(t(v),T(v)-t(v))}$. 
Observe that for $-t(v)\le t \le T(v)-t(v)$, we have $t(g^t v)=t(v)+t$ and $T(g^t v)=T(v)$, so that, 
  for another constant $C\ge 1$, we have
$$
\int_{[0,T(v)]}\tilde{f}(g^{t-t(v)}\tilde{v})dt= C^{\pm 1} e^{\frac{(1-\delta)}{2}d(o,p.o)}\,.
$$

Coming back to (\ref{decomposition-de-la-mesure}), we obtain 
$$
\int_{T^1C_1}f dm^{ps}
= C^{\pm 1}\sum_{p\in \Pi_1} e^{\frac{(1-\delta)}{2}d(o,p.o)}
\int_{J_0\times p.J_0}d\nu_o(v^-)d\nu_o(v^+)\,.
$$
We know that $d\nu_{p^{-1}.o}(\xi)=e^{- \delta \beta_\xi(p^{-1}.o,o)}d\nu_o(\xi)$. 
If $\xi\in J_0$, the triangular inequality implies $\beta_\xi(p.o,o)=d(o,p.o)\pm cst$, so 
$$\nu_o(p.J_0)=(p^{-1}_*\nu_o)(J_0)=\nu_{p^{-1}.o}(J_0)= C^{\pm 1} \nu_o(J_0) e^{- \delta d(o,p.o)}.$$
Our integral is now estimated by
$$
\int_{T^1C_1}f dm^{ps}=C^{\pm 1}\sum_{p\in\Pi_1}e^{\frac{1-3\delta}{2}d(o,p.o)}
$$
We now wish to estimate $d(o,p.o)$. In the upper-half plane, we can assume that $\partial H=\{ z \, : \, Im(z)=1 \}$, then
$\Pi_1=\{ z\mapsto z+n\kappa  \, , \, n\in \N \}$, for some $\kappa>0$. 
Then, using the exact formula for the distance on $\H$, one has
$$\cosh(d(o,p.o))=1+\frac{|o-po|^2}{2Im(o)Im(po)},$$
so
$$
e^{d(o,po)}=\kappa^2 p^2+O(1).
$$

The series $\sum_{p\in\Pi_1}e^{\frac{1-3\delta}{2}d(o,p.o)}$ behaves therefore as 
$\sum_{n\in\N}n^{1-3\delta}$, and converges if and only if $1-3\delta<-1$, that is $\delta>2/3$. This concludes the proof of Theorem \ref{integrability-of-balls}.

\begin{rem} In higher dimension, the only differences in the proof are that 
$\tau(u)\asymp e^{(k-\delta/2)d(u,K_0)}$ if $u$ belongs to a cusp of rank $k$, and 
$\Pi_1\simeq \Z \kappa_1\oplus \Z\kappa_2\dots \oplus\Z\kappa_k$, so that 
the series to estimate is $\sum_{(p_1,\dots,p_k)\in\Z^k}(p_1^2\kappa_1^2+\dots+p_k^2\kappa_k^2)^{\frac{k-3\delta}{2}}$. 
And this series converges iff $\delta>2k/3$.
\end{rem}


\subsection{An almost sure log Cesaro convergence}
\label{proofofalmostsurecesaro}

 The next Lemma will be useful for the case when $S$ has cusps,
as we do not know if $\tau$ is uniformly continuous in that case.
\begin{lem}\label{indephoro} Assume that $S$ is geometrically finite.
 Let $u \in \mathcal{E}$ be a non-periodic vector for the horocyclic flow, and $s_0 \in \R$. Then
 $$\lim_{t\rightarrow +\infty} \frac{\tau(g^th^{s_0}u) }{\tau(g^tu)}=1.$$ 
\end{lem}

\begin{proof} 
 Let $f\geq 0$ be a continuous, compactly supported function with nonzero $m$-integral.
 Let $\epsilon>0$. 
By Theorem \ref{magic-equivalent}, we have that for all $t>0$ sufficiently large,
 $$
e^{-\epsilon} t^\delta \tau(g^th^{s_0}u)\int_\mathcal{E} fdm 
\leq \int_{-t}^t f(h^{s+s_0}u)ds,
$$
 and similarly
 $$
\int_{-t}^t f(h^{s}u)ds\leq 
e^\epsilon t^\delta  \tau(g^t u)\int_\mathcal{E} fdm\,.
$$
Since $\left| \int_{-t}^t f(h^{s+s_0}u)ds-\int_{-t}^t f(h^{s}u)ds \right| \leq 2s_0\|f\|_\infty$, 
we have
$$
e^{-\epsilon} t^\delta \tau(g^th^{s_0}u)\int_\mathcal{E} fdm 
\leq e^{\epsilon}t^\delta  \tau(g^t u)\int_\mathcal{E} fdm + 2s_0\|f\|_\infty. 
$$
 Thus,
$$
\frac{\tau(g^th^{s_0}u)}{\tau(g^t u)} \leq 
e^{2\varepsilon} + \frac{2s_0\| f\|_\infty}{t^\delta \tau(g^t u)\int_\mathcal{E} fdm}.
$$
As $t^\delta \tau(g^t u)\rightarrow +\infty$ when $t\to +\infty$, this proves that 
$\limsup_{t\rightarrow +\infty} \frac{\tau(g^th^{s_0}u)}{\tau(g^t u)}\leq 1$. 
The liminf is obtained by reversing the roles of $u$ and $h^{s_0}u$.
\end{proof}

We now prove Theorem \ref{almost-sure-cesaro}. 
By assumption, $\tau \in L^1(m^{ps})$. Let $f$ be as in the Theorem.
By Birkhoff ergodic theorem, for $m^{ps}$-a.e. $u\in\Omega$, we have
$$
\lim_{\rho \rightarrow +\infty} \frac1{\rho}\int_0^\rho \tau(g^t u)dt= 
\int_\Omega \tau \frac{dm^{ps}}{m^{ps}(\Omega)}.
$$
Thanks to lemma \ref{indephoro}, the set of $u\in T^1S$ such that the above convergence 
holds is saturated by the leaves of the horocyclic flow. 
Thus, for $\hat{\mu}$-a.e. $(u^+,t)\in\mathcal{H}$, and all $u^-\neq u^+$, the convergence
holds for $u=(u^-,u^+,t)$. 
In particular, for $\bar{\mu}$-a.e. $\vu\in\R^2\setminus\{0\}$, the convergence holds 
for $u=\Gamma\Psi(\vu)\in T^1S$. \\

 Let $\epsilon,\eta>0$ be arbitrary, decompose
$f=\sum_i f_i,$ as a finite sum of nonnegative nonzero continuous functions 
such that $v^{(\vu)}(f_i)\leq e^\epsilon$. 
From (\ref{eq: majoration3}), we deduce that for all $T\geq T_0$ large enough,
$$
\frac{\sum_{\gamma \in \Gamma_T}f_i\left(\frac{\gamma \vu}{T}\right)}{T^{(1+\alpha)\delta}}
\leq \frac{2e^{\eta+2\epsilon} \tau(g^{(1-\alpha)\log T-\log r_i+\eta} u)}{m^{ps}(T^1S)} 
\int_{\R^2}  \frac{f_i(\vv)}{(\vv\star \vu)^\delta}d\bar{\mu}(\vv),
$$
where $r_i=r^{(\vu)}(f_i)$.
Integrating over $T$ with respect to the measure $dT/T$, one has
\begin{multline*}
\int_{T_0}^S \frac1{T^{(1+\alpha)\delta}}\sum_{\gamma \in \Gamma_T}f_i\left(\frac{\gamma \vu}{T}\right)\frac{dT}{T}\\
\leq \frac{2 e^{\eta+2\epsilon}}{m^{ps}(T^1S)} \int_{\R^2}  \frac{f_i(\vv)}{(\vv\star \vu)^\delta}d\bar{\mu}(\vv)
\int_{T_0}^S \tau(g^{(1-\alpha)\log T-\log r_i+\eta} u)\frac{dT}{T}.
\end{multline*}
However,
$$
\int_{T_0}^S \tau(g^{(1-\alpha)\log T-\log r_i+\eta} u)\frac{dT}{T}=
\frac1{(1-\alpha)} \int_{(1-\alpha)\log T_0-\log r_i+\eta}^{(1-\alpha)\log S-\log r_i+\eta}\tau(g^t u)dt,
$$
which, for large $S$, is equivalent to $m^{ps}(\Omega)^{-1} \int_\Omega \tau dm^{ps} \log S$, as by assumption $u$ is generic.
 This proves that
 \begin{multline*}
 \limsup_{S\rightarrow +\infty} \frac{1}{\log S} \int_{1}^S \frac1{T^{(1+\alpha)\delta}}\sum_{\gamma \in \Gamma_T}f_i\left(\frac{\gamma \vu}{T}\right)\frac{dT}{T}\\
 \leq 2e^{\eta+2\epsilon}
 \int_\Omega \tau \frac{dm^{ps}}{m^{ps}(\Omega)^2} \int_{\R^2}  \frac{f_i(\vv)}{(\vv\star \vu)^\delta}d\bar{\mu}(\vv).
 \end{multline*}
 The lower bound obtained from (\ref{eq: minoration3}) is similar. 
Summing over $i$ these inequalities, since $\epsilon$ and $\eta$ were arbitrary, 
yields to Theorem \ref{almost-sure-cesaro}.\\
 

\subsection{Other almost sure results for Gibbs measures} 

Remark that the last part of the above proof of theorem \ref{almost-sure-cesaro} is the 
only place where we used an almost sure argument with respect to $m^{ps}$. 
This argument holds verbatim for any invariant ergodic measure which can be decomposed into a family of measures on the horocycles and a transverse measure. 

We apply it here to get theorem \ref{Gibbs}.
Let us first introduce some notations. 
If $\varphi:T^1S\to \R$ is a H\"older map, it is possible to associate to it a $(g^t)$-invariant measure 
$m_\varphi$ on $T^1S$, which shares a lot of properties with $m^{ps}$, 
the Patterson-Sullivan measure being the Gibbs measure associated to any constant potential.  
We refer for example to \cite{SchapETDS} or \cite{Coudene} 
for the construction of such measures on convex-cocompact or geometrically finite manifolds. 
We will just mention that given a potential $\varphi$, one constructs first a family of measures
$(\nu_x^\varphi)_{x\in\H}$ on the limit set $\Lambda\subset\partial\H$, 
which allow to define a product measure 
$m^\varphi$ on $T^1S$, which is ergodic, mixing, and whose support is $\Omega$. 

The measure $m^\varphi$ induces a family 
$(\mu_{H^-}^\varphi)$ of measures on horocycles that vary transversally continuously, and a 
measure $\hat{\mu}^\varphi$ on $\mathcal{H}$ which is $\Gamma$-quasi-invariant, and satisfies 
$d\hat{\mu}^\varphi(\xi,t)=e^{\delta^\varphi t}dt\,d\nu_o^\varphi(\xi)$. 
Therefore, exactly in the same way as in section \ref{defpsi}, it induces a measure 
$\bar{\mu}^\varphi$ on $\R^2\setminus\{0\}$, which can be written as 
$d\bar{\mu}^\varphi(r,\theta)=2r^{2\delta^\varphi-1}dr\,d\bar{\nu}_o^\varphi(\theta)$. 
In contrast with $\hat{\mu}$ and $\bar{\mu}$, the measures $\hat{\mu}^\varphi$ and $\bar{\mu}^\varphi$ 
are not $\Gamma$-invariant, but only quasi-invariant, with an explicit H\"older cocycle. 
We can state:

\begin{thm}\label{Gibbs} Assume that $\Gamma_0$ is
 a nonelementary group containing $-I$ as unique element of torsion, 
and is either convex-cocompact, or geometrically finite with
$\int_{T^1S}\tau\,dm^\varphi<\infty$. 
Write $S=\Gamma_0\backslash \H$. Then,
 with the same notations as in Theorem \ref{th:convexcocompact}, 
we have for $\bar{\mu}^\varphi$-almost every $\vu$,

\eq{logcesaro-gibbs}{ 
\lim_{S\rightarrow +\infty} \frac{1}{\log S}\int_1^S \frac{1}{T^{ \delta}} 
 \sum_{\gamma \in \Gamma_T}f\left( \gamma \vu \right) \frac{dT}{T}
= \frac{2\int_{T^1S} \tau dm^{\varphi}}{m^\varphi(T^1S)m^{ps}(T^1S)}
\int_{\R^2} \frac{f(\vv)}{(\vv \star \vu)^{\delta}}d\bar{\mu}(\vv).   
}
The function $\tau$ is the same as in Theorem \ref{magic-equivalent}.
\end{thm}

Remark that the condition $\int_{T^1S}\tau\,dm^\varphi<\infty$ implies that $m^\varphi$ is finite, 
which is not always the case on geometrically finite hyperbolic surfaces 
(see \cite{Coudene} for conditions ensuring it).

\begin{proof} We do not give additional details,
 it is enough to replace $m^{ps}$ with $m^\varphi$, $\bar{\mu}$ 
with $\bar{\mu}^\varphi$, and $\hat{\mu}$ with $\hat{\mu}^\varphi$
in the end of the proof of theorem \ref{almost-sure-cesaro} above. 
\end{proof}


\section{Large scale}\label{6}

\subsection{Sketch of the proof}
The full proof of theorem \ref{largescale} is quite technical 
in the case of a general norm on $M(2,\R)$, so that we try to present the ideas
to the reader. 

$*$ {\em First step : relate the sum 
$\sum_{\gamma\in \Gamma_T} f(\frac{\gamma\vu}{T})$, for $f$ continuous
 with compact support in $\R^2\setminus\{0\}$,
 to an integral of $\bar{f}$ along a horocycle. } \\
This is done in lemma \ref{relation-horocycle-orbit}. Heuristically, 
the above sum is comparable to an integral $\int_{I(\vu,f).T^2} \bar{f}(h^s g^{-2\log T} u)\,ds$, 
where $I(\vu,f)$ is an interval of the form $(-\Theta^m(\vu,f)-\Theta(\vu,f), -\Theta^m(\vu,f)+\Theta(\vu,f))$. \\

$*$ {\em Second step : conclude the proof for $f$ continuous with compact support in a bounded 
"disk" $\mathcal{D}(\vu)$ in $\R^2\setminus\{0\}$}\\
This is an almost immediate consequence of the first step combined 
with theorem \ref{magic-equivalent-flowed}, which allows
to compute the limit of the above integral. \\

$*$ {\em Third step : prove theorem \ref{largescale} 
for $f$ continuous with  support in $\R^2\setminus\{0\}$. }\\
This follows from the fact (see lemma \ref{probadisque}) that for $T$ large, 
all $\frac{\gamma\vu}{T}$  belong to the "disk" $\mathcal{D}(\vu)$. \\

$*$ {\em Fourth step : prove theorem \ref{largescale}, in the case of the $l^2$-norm}\\
If $f:\R^2\to \R$ is continuous, and its support contains $0$, 
 the only difficulty is to understand the behaviour of $\Gamma_T$ in 
a neighbourhood of $0$. We compute the mass of the limit measure obtained for 
continuous functions with support in $\R^2\setminus\{0\}$  (lemma \ref{probability}), and 
the cardinal of $\Gamma_T$ (lemma \ref{cardinal-orbit}). 
These lemmas allow us to deduce 
 a result of tightness of the probability measures 
$\frac{1}{|\Gamma_T|}\sum_{\gamma\in\Gamma_T}\delta_{\gamma\vu/T}$. 
These measures do not loose mass in the neighbourhood of $0$,
 and this allows to deduce theorem \ref{largescale}
for all continuous functions on $\R^2$. \\

$*$ {\em Last step : prove theorem \ref{largescale} for a general strictly convex norm}\\
Steps 1 to 3 apply for all strictly convex norms on $\SL(2,\R)$. The only thing to prove is to deduce the tightness
in the case of a general norm from the above tightness result in the case of the $l^2$-norm.

\subsection{The maps $\Theta$ and the set $\mathcal{D}(\vu)$}\label{applications-Theta}
The initial vector $\vu$ 
is fixed for the entire section. 
Define
 $$
\kappa(\vu,\vv,s)=  
\left\| \Psi(\vv) \left( \begin{array}{cc} 1 & s \\ 0 & 0 \end{array}\right) \Psi(\vu)^{-1}\right\|.
$$
Thanks to (\ref{eq: propPsi}), for all $\lambda>0$, we have 
 $\kappa(\vu,\lambda \vv,s)=\lambda\kappa(\vu,\vv,s)$.

 Given $\vu,\vv$, by convexity of the chosen norm, 
 the set $\{s\in\R,\,\kappa(\vu,\vv,s)\leq 1\}$ is either empty or is a
 compact interval denoted by $[ \Theta^-(\vu,\vv),\Theta^+(\vu,\vv)]$.
 We use  the convention that $\Theta^+(\vu,\vv)=\Theta^-(\vu,\vv)=0$ when the interval is empty.

 We  denote by $\Theta^m(\vu,\vv)=\frac12(\Theta^-(\vu,\vv)+\Theta^+(\vu,\vv))$ the middle of this interval, and 
 $\Theta(\vu,\vv)=\frac{1}2(\Theta^+(\vu,\vv)-\Theta^-(\vu,\vv))$ its half-length. 

 Let us also define
 $$
\mathcal{D}(\vu)=\{\vv \in \R^2\setminus \{0\}; \Theta(\vu,\vv)>0\}\,,\quad \mbox{and}\quad
\mathcal{D}_0(\vu)=\mathcal{D}(\vu)\cup\{0\}\,.
$$
 In general, one can check that these are non-empty, open, bounded sets of $\R^2$.  
The set $\mathcal{D}_0(\vu)$ can also be described as
 $\mathcal{D}_0(\vu)=\{\vv \in\R^2\,,\, \exists s\in\R,\, \kappa(\vu,\vv,s)<1\}$. 
Note also that, for $\lambda>0$, $\mathcal{D}(\lambda \vu)=\lambda \mathcal{D}(\vu)$,
 and that $\mathcal{D}_0(\vu)$ is a star-shaped set from the point $0$. \\

In the case where $\|.\|$ (resp. $|.|$) is the  $l^2$-norm on $\SL(2,\R)$ (resp. on $\R^2$),
explicit computations give
$$
\kappa(\vu,\vv,s)^2= \frac{|\vv|^2}{|\vu|^2}+s^2|\vv|^2|\vu|^2,
$$
whence we deduce that for $|\vv|\leq |\vu|$,
$$
\Theta^\pm(\vu,\vv)=\frac{\pm 1}{|\vu|.|\vv|}\sqrt{1-\frac{|\vv|^2}{|\vu|^2}},\quad \Theta^m(\vu,\vv)=0,\quad
\mbox{and}\quad
\Theta(\vu,\vv)=\frac{1}{|\vu|.|\vv|}\sqrt{1-\frac{|\vv|^2}{|\vu|^2}}. 
$$
and all these quantities equal $0$ when $|\vv|>|\vu|$. 
Thus, we also have $\mathcal{D}(\vu)=D(0,|\vu|)\setminus\{0\}$, and $\mathcal{D}_0(\vu)=D(0,|\vu|)$.

 In full generality,  the maps 
$\vv\mapsto\Theta^\pm(\vu,\vv)$ may not be continuous on $\mathcal{D}(\vu)$.
 However, it is easy to see that $\Theta^+$ (resp. $\Theta^-$) 
is always upper semi-continuous (resp.  lower semi-continuous),  
and that if one of the functions 
$\Theta^\pm(\vu,.)$ is not continuous at $\vv$, then the set $\{s\,:\, \kappa(\vu,\vv,s)=1 \}$ 
contains an interval. Observe that $\kappa$ is the norm of a matrix which is an affine function of $s$.
 So the set $\{s\,:\, \kappa(\vu,\vv,s)=1 \}$ 
can contain an interval only if the unitary sphere of $\|.\|$ contains a line segment, so
 that the norm is not strictly convex. 
From now on, we assume 
that it does not happens, so that all  maps $\Theta^{\pm,m}$ defined above are continuous on $\mathcal{D}(\vu)$.

We have to think to $\vv$ as a vector varying in the small support of a continuous function $f$ with compact support
$f:\mathcal{D}(\vu)\to \R$. Therefore,  
we introduce also the following functions. 
Fix a small parameter $\sigma>0$ such that $e^{\pm\sigma}\supp(f)\subset \mathcal{D}(\vu)$, and define
 $$
\overline{\Theta}_\sigma^{+}(\vu,f)=\sup_{ \vv \in e^{\pm\sigma}\supp(f) } \Theta^{+}(\vu,\vv), \; \; 
 \overline{\Theta}_\sigma^{-}(\vu,f)=\sup_{ \vv \in e^{\pm\sigma}\supp(f) } \Theta^{-}(\vu,\vv),
$$
 $$
\underline{\Theta}_\sigma^{+}(\vu,f)=\inf_{ \vv \in e^{\pm\sigma}\supp(f) } \Theta^{+}(\vu,\vv), \; \; 
\underline{\Theta}_\sigma^{-}(\vu,f)=\inf_{ \vv \in e^{\pm\sigma}\supp(f) } \Theta^{-}(\vu,\vv),
$$

 $$\overline{\Theta}^{m}_\sigma(\vu,f)=\frac12(\overline{\Theta}_\sigma^+(\vu,f)+\underline{\Theta}_\sigma^-(\vu,f)), \; \; 
\overline{\Theta}_\sigma(\vu,f)=\frac{1}2(\overline{\Theta}_\sigma^+(\vu,f)-\underline{\Theta}_\sigma^-(\vu,f)),
$$
$$
\underline{\Theta}^{m}_\sigma(\vu,f)=\frac12(\underline{\Theta}_\sigma^+(\vu,f)+\overline{\Theta}_\sigma^-(\vu,f)), \; \; 
\underline{\Theta}_\sigma(\vu,f)=\frac{1}2(\underline{\Theta}_\sigma^+(\vu,f)-\overline{\Theta}_\sigma^-(\vu,f)).
$$
 
By definition, for all $\vv \in e^{\pm\sigma}\mbox{supp}(f)$, we have
$$
\left(\overline{\Theta}^-_\sigma(\vu,f),\underline{\Theta}^+_\sigma(\vu,f)\right)\subset 
\left(\Theta^-(\vu,\vv),\Theta^+(\vu,\vv)\right)
\subset \left(\underline{\Theta}^-_\sigma(\vu,f),\overline{\Theta}^+_\sigma(\vu,f)\right)\,,
$$
or equivalently
$$
(\underline{\Theta}^m_\sigma(\vu,f)\pm\underline{\Theta}_\sigma(\vu,f)) 
\subset \left(\Theta^-(\vu,\vv),\Theta^+(\vu,\vv)\right)\subset
(\overline{\Theta}^m_\sigma(\vu,f)\pm\overline{\Theta}_\sigma(\vu,f)) \,.
$$

 By continuity of $\Theta^\pm$ on $\mathcal{D}_0(\vu)$, 
given $\eta>0$, we can find $\alpha>0$, such that for all $\vv_0\in \mathcal{D}(\vu)$, and 
all continuous functions with compact support in $B(\vv_0,\alpha)$
 with $e^{\pm \sigma} \mbox{supp}(f)\subset\mathcal{D}(\vu)$, 
these functions satisfy for all $\vv\in\mbox{supp}(f)$, 
$|\Theta^\pm(\vu,f)-\overline{\Theta}^\pm_\sigma(\vu,\vv)|\le \eta$, and similar approximations for  
$\overline{\Theta}_\sigma^m, \underline{\Theta}_\sigma^m,\underline{\Theta}_\sigma,\overline{\Theta}_\sigma$.


\subsection{Relation between $\Gamma_0$-orbits on $\R^2\setminus\{0\}$ and integrals along horocycles} 
 
 Fix a continuous function $f$ of compact support on
 $\mathcal{D}(\vu)$. Recall that $\vu\neq 0$ is fixed.
 
 \begin{lem}\label{relation-horocycle-orbit} Let $\sigma>0$. 
For all $T$ sufficiently large, and all $\gamma \in\Gamma_0$ such that $\frac{\gamma \vu}T \in \supp(f)$,
 \begin{itemize}
 \item If $||\gamma||\leq T$, then 
  $\displaystyle 
\int_{-\overline{\Theta}_\sigma^+(\vu,f)T^2-1}^{-\underline{\Theta}_\sigma^-(\vu,f)T^2+1} 
\tilde{f}(\gamma \Psi(\vu)a_{-2\log T} n_s)ds=f\left(\frac{\gamma \vu}{T}\right),
 $
 \item If $||\gamma||> T$, then
  $\displaystyle 
\int_{-\underline{\Theta}_\sigma^+(\vu,f)T^2+1}^{-\overline{\Theta}_\sigma^-(\vu,f)T^2-1}
 \tilde{f}(\gamma \Psi(\vu)a_{-2\log T} n_s)ds=0\,. 
$
 \end{itemize}
 \end{lem}

 \begin{proof}
  Let $\vv=\gamma \vu/T \in \supp(f)$. 
We will write simply $c$ for $c_{\vu/T}(\gamma)$.
  
  By definition of $\tilde{f}$, for all interval $I\subset \R$, the integral
  $\int_I \tilde{f}(\Psi(\vv)n_s)ds$ is equal to $f(\vv)$ if $[-1,1]\subset I$,
 and is equal to zero if $I\cap [-1,1]=\emptyset$.
  By definition of the cocycle $c$, we have 
$$
\Psi(\vv)=\gamma \Psi(\vu/T)n_{-c}=\gamma \Psi(\vu)a_{-2\log T}n_{-c},
$$
  so that for all intervals $J\subset\R$, the integral  
$\int_{J} \tilde{f}(\gamma \Psi(\vu)a_{-2\log T} n_s)ds$ equals  $f(\vu/T)$ 
if $[-1,1]-c\subset J$, and $0$ if $([-1,1]-c)\cap J=\emptyset$.

We now estimate the size of the cocyle in terms of $||\gamma||$. 
By definition of $c=c_{\vu/T}(\gamma)$, we have
$\displaystyle 
\gamma=\Psi(\vv)\left( \begin{array}{cc} 1 & c\\  0 & 1  \end{array}\right) \Psi(\vu/T)^{-1}.$
Therefore, 
$\displaystyle 
\gamma=\Psi(\vv)\left( \begin{array}{cc} 1 & c \\  0 & 1  \end{array}\right) 
  a_{2\log T} \Psi(\vu)^{-1},$ so that 
$$
\gamma=T \Psi(\vv)\left( \begin{array}{cc} 1 & T^{-2} c \\  0 & 0  \end{array}\right) 
   \Psi(\vu)^{-1} + T^{-1} \Psi(\vv)\left( \begin{array}{cc} 0 & 0 \\ 0 & 1  \end{array}\right) 
   \Psi(\vu)^{-1}.
$$
Note that the term $\Psi(\vv)\left( \begin{array}{cc} 0 & 0 \\ 0 & 1  \end{array}\right) 
\Psi(\vu)^{-1}$ is a bounded matrix for all $\vv \in \supp(f)$. 
For $T$ large, the second term on the right-hand
side is really small, at least compared to $\|\gamma\|$, 
which is bounded from below by a positive constant uniformly on $\Gamma_0$. 
Thus, we have for all large $T$, \eq{bornegamma}
  {
  \|\gamma\|=e^{\pm \sigma} T\kappa(\vu,\vv,T^{-2} c). 
}  
  Now, assume that $\|\gamma\|\leq T$. Then
  $\kappa(\vu,\vv,T^{-2}c)\leq e^\sigma,$
  so $\kappa(\vu,e^{-\sigma}\vv,T^{-2}c)\leq 1,$ 
and by definition of the maps $\Theta$, this implies
$$
\underline{\Theta}_\sigma^-(\vu,f) \leq T^{-2} c \leq \overline{\Theta}_\sigma^+(\vu,f).
$$ 
  This proves the first point, as 
$[-1,1]-c\subset [-T^2\overline{\Theta}_\sigma^+(\vu,f)-1, -T^2\underline{\Theta}_\sigma^-(\vu,f)+1]$.\\

  For the second point, assume that $\|\gamma\|>T$, then $\kappa(\vu,\vv,T^{-2}c)>e^{-\sigma}$,
  so $\kappa(\vu,e^\sigma\vv,T^{-2}c)> 1$, thus either
  $T^{-2}c>\underline{\Theta}_\sigma^+(\vu,f)$, either $T^{-2}c<\overline{\Theta}_\sigma^-(\vu,f)$.
 In any case, the set $[-1,1]-c$ does not intersect the interval 
$[-T^2\underline{\Theta}_\sigma^+(\vu,f)+1, -T^2\overline{\Theta}_\sigma^-(\vu,f)-1]$.
\end{proof}


\subsection{Proof of theorem \ref{largescale}
for functions with  compact support in $\mathcal{D}(\vu)$}
 
As in the proof of Lemma \ref{estimeelemme}, Lemma \ref{relation-horocycle-orbit} implies the estimates, 
for a function $f$ of sufficiently small support:
$$
 \sum_{\gamma \in \Gamma_T} f\left( \frac{\gamma \vu}{T}\right) \leq 
2\int_{-\overline{\Theta}_\sigma^+(\vu,f)T^2-1}^{-\underline{\Theta}_\sigma^-(\vu,f)T^2+1} \bar{f}(h^s g^{-2\log T}u)ds,
$$
and 
$$
 \sum_{\gamma \in \Gamma_T} f\left( \frac{\gamma \vu}{T}\right) \geq 
2\int_{-\underline{\Theta}_\sigma^+(\vu,f)T^2+1}^{\overline{\Theta}_\sigma^-(\vu,f)T^2-1}
 \bar{f}(h^s g^{-2\log T}u)ds.
$$

Consider the first integral; if we neglect the $\pm 1$ in its bounds and 
 translate the interval of integration by $T^2\overline{\Theta}^m_\sigma(\vu,f)$, we get: 
$$
\int_{-\overline{\Theta}_\sigma^+(\vu,f)T^2-1}^{-\underline{\Theta}_\sigma^-(\vu,f)T^2+1} \bar{f}(h^s g^{-2\log T}u)ds
\le 2\|f\|_\infty+
\int_{-T^2\overline{\Theta}_\sigma(\vu,f)}^{T^2\overline{\Theta}_\sigma(\vu,f)} 
\bar{f} (g^{-2\log T}h^{-\overline{\Theta}^m_\sigma(\vu,f)}u)\,ds\,.
$$
Note that, if $t=T^2\overline{\Theta}_\sigma(\vu,f)$, 
$g^{-\log t}(g^{\log \overline{\Theta}_\sigma(\vu,f)}h^{-\overline{\Theta}^m_\sigma(\vu,f)} u)=
g^{-2\log T}h^{-\overline{\Theta}^m_\sigma(\vu,f)}u$,
Apply Theorem \ref{magic-equivalent-flowed}, 
for  $t$ and $u'=g^{\log \overline{\Theta}_\sigma(\vu,f)}h^{-\overline{\Theta}^m_\sigma(\vu,f)} u$.
For $T$ large enough, it gives 
$$
\frac{1}{T^{2\delta}}\sum_{\gamma\in\Gamma_T}f(\frac{\gamma\vu}{T})\le \frac{2\|f\|_\infty}{T^{2\delta}}+
\frac{2e^\sigma \overline{\Theta}_\sigma(\vu,f)^\delta \tau(g^{\log\overline{\Theta}_\sigma(\vu,f)}h^{-\overline{\Theta}^m_\sigma(\vu,f)}u)}{m^{ps}(T^1S)}
\int_{T^1S}\bar{f}\,dm\,.
$$
We have $\int_{T^1S}\bar{f}\,dm=\int_{\R^2}f\,d\bar{\mu}$. 
As all functions $\Theta$ are continuous, for a given $\eta>0$, if the support of $f$ is small enough
(see the end of the above section), we get, for $T$ large enough,  
\eq{largescalesup}{
\frac{1}{T^{2\delta}}\sum_{\gamma\in\Gamma_T}f(\frac{\gamma\vu}{T})\le
\frac{2e^\sigma e^\eta}{m^{ps}(T^1S)}
\int_{\R^2}\Theta(\vu,\vv)^\delta\,\tau(g^{\log\Theta(\vu,\vv)}h^{-\Theta^m(\vu,\vv)}\Psi(\vu))f(\vv)d\bar{\mu}(\vv)\,.
}
The same reasoning with the lower bound gives 
\eq{largescaleinf}{
\frac{1}{T^{2\delta}}\sum_{\gamma\in\Gamma_T}f(\frac{\gamma\vu}{T})\ge
\frac{2e^{-\sigma} e^{-\eta}}{m^{ps}(T^1S)}
\int_{\R^2}\Theta(\vu,\vv)^\delta\,\tau(g^{\log\Theta(\vu,\vv)}h^{-\Theta^m(\vu,\vv)}\Psi(\vu))f(\vv)d\bar{\mu}(\vv)\,.
}

Consider now a continuous, non-negative, symmetric
 function $f$ whose support is a subset of $\mathcal{D}(\vu)$.

 Let $\sigma>0$ be very small. Using a partition of unity, 
we can write $f=\sum_i f_i$ of finitely many functions of small support
 for which the previous paragraph applies.
 Then, adding the inequalities (\ref{eq: largescalesup}) and (\ref{eq: largescaleinf})
 over $i$, and since $\sigma,\eta$ were arbitrary, we obtain
 the desired result for $f$ compactly supported in $\mathcal{D}(\vu)$.


\subsection{Proof of theorem  \ref{largescale}  for functions with    support in $\mathcal{D}_0(\vu)$}
Now, $f$ is continuous, nonnegative, and  supported in $\R^2\setminus\{0\}$. 

\begin{lem} \label{probadisque}
For any compact neighbourhood $W$ of $\overline{\mathcal{D}(\vu)}$, 
there exists $T_0>0$ such that for all $T\ge T_0$, 
and all $\gamma \in \Gamma_T$, we have $\gamma \vu/T \in W$. 
\end{lem}

\begin{proof}
 For all $\sigma>0$ and sufficiently large $T$, Equation (\ref{eq: bornegamma}) 
implies that for all $\gamma \in \Gamma_T$ and $\vv=\gamma \vu/T$,
$c=c_{\vu/T}(\gamma)$, we have 
 $$
\kappa(\vu,\vv,T^{-2}c)\leq e^\sigma,
$$
that is
$$
\kappa(\vu,e^{-\sigma}\vv,T^{-2}c)\leq 1,
$$
 so that $e^{-\sigma}\vv \in \overline{\mathcal{D}(\vu)}$. 
If $\sigma$ is small enough so that $e^{\pm \sigma}\overline{\mathcal{D}(\vu)}\subset W$, this proves the claim.
\end{proof}

 Let $\epsilon>0$. 
Write
 $$
\Xi(\vu,\vv)=\Theta(\vu,\vv)^\delta \tau(g^{\log \Theta(\vu,\vv)}h^{-\Theta^m(\vu,\vv)} \Psi(\vu)),
$$
Choose a compact neighbourhood $Q$ of $\partial \mathcal{D}_0(\vu)$, and a neighbourhood $W$ of 
$\mathcal{D}(\vu)$ in $\R^2\setminus\{0\}$ such that $W\subset \mathcal{D}(\vu)\cup Q^\circ$. 
 Assume that $Q$ is 
thin enough so that $Q \subset  \mathcal{D}(2\vu)$ and 
$$
\int_{Q} \Xi(2\vu,\vv)d\bar{\mu}<\epsilon.
$$ 
 This is possible because $\bar{\mu}(\partial \mathcal{D}_0(\vu))=0$, as $\mathcal{D}_0(\vu)$ 
is a star-shaped set around $0$, its boundary is the graph of a function of the angle, and $\bar{\mu}$ is a product measure in polar
 coordinates.
 
Consider a continuous, nonnegative function $f:\R^2\setminus\{0\}\to \R$. 
We can decompose $f=f_\mathcal{D}+f_Q+f_\infty$ as a sum of continuous functions,
 where $f_\mathcal{D}$ has compact support 
in $\mathcal{D}(\vu)$, $f_Q$ has compact support in $Q$ and 
is bounded by $\sup_Q f$, and $f_\infty$ has support outside $W$.

We already know that theorems \ref{largescale} and  \ref{largescale} bis apply to $f_\mathcal{D}$. 

By the previous Lemma, for $T$ sufficiently large, 
$$
\sum_{\gamma \in \Gamma_T} f_\infty(\gamma \vu/T)=0,
$$
so that we need only to show that the sum $\sum_{\gamma \in \Gamma_T} f_Q(\gamma \vu/T)$ is small.
 However, $\mbox{supp}(f_Q)\subset Q \subset 2\mathcal{D}(\vu)=\mathcal{D}(2\vu)$. 
So we can apply the result for $f_Q$ and $2\vu$, which gives 
$$
 \lim_{T\rightarrow +\infty} \frac1{2^{2\delta}T^{2\delta}}\sum_{\gamma \in \Gamma_{2T}} f_Q\left(\frac{\gamma \vu}{T}\right)=\int_{Q} \Xi(2\vu,\vv)f_Q(\vv)d\bar{\mu}(\vv),
$$
which is smaller than $\epsilon \sup_Q f$. Thus,
$$
\frac1{T^{2\delta}}\sum_{\gamma \in \Gamma_{T}} f_Q\left(\frac{\gamma \vu}{T}\right)\leq 
2^{2\delta} \epsilon \sup_Q f.
$$
This proves that theorem \ref{largescale}  holds for all continuous 
functions $f$  supported in $\R^2\setminus \{0\}$.


\subsection{Proof of theorem \ref{largescale} in the case of the $l^2$-norm}

\begin{lem}\label{probability}
 Assume that $|\vu|=1$ and that the norm $\|.\|$ is the $l^2$-norm. Then the integral
$$
I=\int_{\mathcal{D}(\vu)} \Xi(\vu,\vv) d\bar{\mu}(\vv),
$$
is equal to $\frac{1}{\delta}$.
\end{lem}

\begin{proof}
In the case of $l^2$-norm, ( $\Theta^m=0$,) $\Theta(\vu,\vv)=r^{-1}\sqrt{1-r^2}$ 
where $r=|\vv|$, $\mathcal{D}(\vu)=D(0,1)\setminus\{0\}$. 
So the integral is equal to 
$$
I = \int_{0}^1 \int_{\theta \in S^1} \mu_{H^-(u)}\left((h^sw)_{|s|\le \Theta(r)}\right)2r^{2\delta-1}drd\bar{\nu}_o(\theta).
$$
Observe that the quantity to integrate in the variable $v$ depends only on $r=|v|$, and not on $\theta$. As
$\bar{\nu}_0$ is a probability measure on $S^1$, we can forget it.
Since $|\vu|=1$, $o=\pi(u)$, and for all $s \in \R$, we can compute 
the Busemann cocycle in the upper-half-plane model: 
$$
e^{-\beta_{(h^s u)^-}(\pi(h_s u),o)}=s^2+1.
$$ 
Thus, using the fact that
$d\mu_{H^-(u)}=e^{-\delta \beta_{(h^s u)^-}(\pi(h^s u),o)}d\nu_o$, 
$$
I = \int_{0}^1 \int_{-\Theta(r)}^{\Theta(r)} 2r^{2\delta-1}(s^2+1)^{\delta}d\nu_o((h^su)^-)dr.
$$
The set of integration is given by 
$\{(r,s)\,:\, |s|\leq \Theta(r)\}=\{(r,s)\, : r^2\leq 1/(s^2+1)\,\}$, so by Fubini
$$
I= \int_\R \int_0^{(s^2+1)^{-1/2}} 2r^{2\delta-1}(s^2+1)^\delta dr d\nu_o((h^su)^-),
$$
that is
$$
I= \int_\R \left( \int_0^{(s^2+1)^{-1/2}} 2r^{2\delta-1}dr \right) (s^2+1)^{\delta} d\nu_o((h^su)^-),
$$
so $I=\frac1\delta \nu_o(\partial \H-\{u^+\})=\frac1\delta$.
\end{proof}

For any fixed norm, the integral considered in the previous Lemma does in fact
 not depend on $\vu$ (this is a corollary of Theorem \ref{largescale} applied to the constant function 1); 
we are however unable to prove this directly for others norms than the $l^2$-norm.\\

\begin{lem}\label{cardinal-orbit} When the norm is the $l^2$ norm, 
the counting function $|\Gamma_T|$ has the following asymptotic.
$$
|\Gamma_T|\sim_{T\rightarrow +\infty} \frac{2T^{2\delta}}{\delta m^{ps}(T^1S)}.
$$
\end{lem} 

\begin{proof}
 When using the $l^2$-norm, we get 
 $\displaystyle 
||\gamma||^2=2\cosh(d(o,\gamma o)),$
 thus
 $$
|\Gamma_T|=|\{\gamma \in\Gamma_0 \, : \, ||\gamma||\leq T \}| = 
2|\{\gamma \in \Gamma \, : \, d(o,\gamma o) \leq \cosh^{-1}(T^2/2) \}|.
$$
 By \cite[Thm 4.1.1]{Roblin}, the counting function has the following asymptotic
 $$
|\{\gamma \in \Gamma \, : \, d(o,\gamma o) \leq t \}|\sim_{t\rightarrow +\infty}
 \frac{e^{\delta t}}{\delta m^{ps}(T^1 S)},$$
 which implies the result.
\end{proof}

 Now, assuming the norm is the $l^2$-norm, let $\nu$ be any weak limit
 of the sequence of probability measures
 $\nu_T =\frac{1}{|\Gamma_T|}\sum_{\gamma \in \Gamma_T} \delta_{\gamma\vu/T}$, where $\delta_x$ 
is the Dirac mass at the point $x$.
By lemma \ref{probadisque},  $\nu$ is 
a probability measure 
supported by $\overline{\mathcal{D}_0(\vu)}$.
 Let $\nu'$ be the measure
 $$
d\nu'=\delta \Xi(\vu,.)d\bar{\mu}.
$$
We have seen (lemma \ref{probability}) that $\nu'$ is a probability,
 and we know that for $f$ continuous, and 
 compactly supported in $\R^2\setminus\{0\}$,
$$
\nu(f)=\lim_{T\rightarrow +\infty} \nu_T(f)=
\frac{\delta m^{ps}(T^1S)}{2}\frac{2}{m^{ps}(T^1S)}\int \Xi f d\bar{\mu}=\nu'(f).
$$
Since $\nu$ and $\nu'$ are probabilities, $\nu(\{0\})=0$, so $\nu=\nu'$, 
which concludes the proof in the case of the $l^2$-norm.


\subsection{Proof of theorem \ref{largescale} for an arbitrary norm}
 For an arbitrary - strictly convex -  norm, we have to show that the measures 
$\nu_T$ do not accumulate around zero, that is, for all $\epsilon>0$, 
there is a neighbourhood $W$ of $0$ such that for large $T$,
 $\sum_{\gamma \in \Gamma_T} 1_W(\gamma \vu/T) < \epsilon T^{2\delta}$.
 Denote by $\Gamma_T^{l^2}$ the set of matrices of norm less than 
$T$ for the $l^2$-norm; there exists $c>0$ such 
that $\Gamma_T\subset \Gamma_{cT}^{l^2}$. 
Now take $W$ such that 
$\sum_{\gamma \in \Gamma_T^{l^2}} 1_{W/c}(\gamma \vu/T) < \epsilon T^{2\delta}$,
 we have
 $$
\sum_{\gamma \in \Gamma_T} 1_W(\gamma \vu/T) \leq \sum_{\gamma \in \Gamma_{cT}^{l^2}} 1_W(\gamma \vu/T)
=\sum_{\gamma \in \Gamma_{cT}^{l^2}} 1_{W/c}(\gamma \vu/(cT))<\epsilon c^{2\delta}T^{2\delta},
$$
 as required.
 

\section{Acknowledgments}

 The first named author wish to thank the Bernoulli Center at EPFL for its hospitality.
The second author benefited from the ANR grant ANR-10-JCJC 0108 during the redaction of this article.




\begin{thebibliography}{99}

\bibitem[Ba1]{mbab2} Babillot, Martine 
{\em On the mixing property for hyperbolic systems} (2002) Israel J. Math. {\bf 129}, 61-76.

\bibitem[Ba2]{Babillot} Babillot, Martine, {\em  Points entiers et groupes discrets: de l'analyse aux systèmes dynamiques} (French) [Lattice points and discrete groups: from analysis to dynamical systems] With an appendix by Emmanuel Breuillard. Panor. Synthèses, {\bf 13}, Rigidité, groupe fondamental et dynamique, 1–119, Soc. Math. France, Paris, 2002. 

\bibitem[Ba-L]{BL} Babillot, Martine; Ledrappier, Fran\c cois, {\em Geodesic paths and 
horocycle flox on abelian covers}, Proc. International Colloquium on Lie groups and Ergodic theory,Tata
Institute of Fundamental Research, Narosa Publishing House, New Delhi (1998), 1-32. 



\bibitem[Bo]{bowditch} Bowditch, Brian H. 
{\em Geometrical finiteness with variable negative curvature}, 
Duke Math. J. {\bf 77} n.1 (1995) 229-274.




\bibitem[Bu]{Burger} Burger, Marc 
{\em Horocycle flow on geometrically finite surfaces}, 
Duke Math. J. {\bf 61}, n.3, (1990) 779-803.



\bibitem[C]{Coudene} Coudene, Yves, {\em Gibbs measures on negatively curved manifolds}. 
J. Dynam. Control Systems {\bf 9} (2003), no. 1, 89–101. 

\bibitem[D]{Dalbo} Dal'bo, Fran\c coise 
{\em Topologie du feuilletage fortement stable}. 
Ann. Inst. Fourier
(Grenoble) {\bf 50} (2000), no. 3, 981--993.

\bibitem[D-O-P]{DOP} Dal'bo, Fran\c coise; Otal, Jean-Pierre; Peign\'e, Marc 
{\em S\'eries de Poincar\'e des groupes g\'eom\'etriquement finis}. 
Israel J. Math. {\bf 118} (2000), 109--124.


\bibitem[Da-S]{DS} Dani, S. G.; Smillie, John 
{\em Uniform distribution of horocycle orbits for Fuchsian groups}. 
Duke Math. J. {\bf 51} (1984), no. 1, 185--194.
 




\bibitem[Fi]{Fisher} Fisher, Albert 
{\em Integer Cantor sets and an order-two ergodic theorem},
Ergodic Theory Dynam. Systems {\bf 13} (1993), no. 1, 45–64. 

\bibitem[F]{Furstenberg} Furstenberg, Harry 
{\em The unique ergodicity of the horocycle flow.}
 Recent advances in topological dynamics (Proc. Conf., Yale Univ., New
Haven, Conn., 1972; in honor of Gustav Arnold Hedlund), 
pp. 95--115. Lecture Notes in Math., {\bf 318}, Springer, Berlin, 1973. 

\bibitem[G]{g1} Gorodnik, Alexander
{\em Uniform distribution of orbits on spaces of frames.} 
Duke Math. J. {\bf 122}  (2004) no. 3 549-589.

\bibitem[GW]{gw} Gorodnik, Alexander; Weiss, Barak
{\em Distribution of lattice orbits on homogeneous varieties,} 
Geom. Func. An. {\bf 17}  (2007) 58-115.

\bibitem[K]{Kim} Kim, Inkang, {\em 
Counting, Mixing and Equidistribution of horospheres in geometrically finite rank one locally symmetric manifolds}, preprint  {\tt arXiv:1103.5003}. 







\bibitem[L1]{led} Ledrappier, Fran\c cois
{\em Distribution des orbites des r\'eseaux sur le plan r\'eel}.
C.R. Acad. Sci. Paris Sr. I Math. {\bf 329} no. 1 (1999) 61-64.

\bibitem[L2]{led2} Ledrappier, Fran\c cois
{\em Ergodic properties of some linear actions}.
Pontryagin conference, 8,, Topology (Moscow, 1998).
J. Math. Sci. (New York) {\bf 105} no. 2 (2001) 1861-1875.

\bibitem[L-P1]{lp1} Ledrappier, Fran\c cois; Pollicott, Mark
{\em Ergodic properties of linear actions of $(2\times 2)$-matrices}.
Duke Math. J. {\bf 116} no. 2 (2003) 353-388.

\bibitem[L-P2]{lp2} Ledrappier, Fran\c cois; Pollicott, Mark
{\em Distribution results for lattices in $\SL(2,\Q_p)$}.
Bull. Braz. Math. Soc. (N.S.) {\bf 36} no. 2 (2005) 143-176.


\bibitem[N1]{no} A.~Nogueira,
{\em Orbit distribution on $\mathbb{R}^2$ under the natural action of
$\hbox{\rm SL}(2,\mathbb{Z})$}. 
Indag. Math. (N.S.)  {\bf 13}  (2002),  no. 1, 103-124.

\bibitem[N2]{no2} A.~Nogueira,
{\em Lattice orbit distribution on $\mathbb{R}^2$},
Ergodic Theory and Dynamical Systems {\bf 30} (2010), no. 4, 1201-1214.

\bibitem[M]{mau} Maucourant, Fran\c cois
{\em Homogeneous asymptotic limits of Haar measure of semisimple linear groups and their lattices},
Duke Math. J. {\bf 136} (2007), no. 2, 357-399.

\bibitem[M-W]{MW} Maucourant, Fran\c cois, Weiss; Barak {\em Lattice actions on the plane revisited}, 
to appear in Geometriae Dedicata


\bibitem[O-S]{OS} Oh, Hee; Shah, Nimish 
{\em Equidistribution and counting for orbits of geometrically finite hyperbolic groups}, preprint.

\bibitem[P-P]{PP} Paulin, Fr\'ed\'eric; Parkonnen, Jouni, {\em Counting arcs in negative curvature}, preprint
  hal-00676941, arXiv:1203.0175.

\bibitem[Po]{Pollicott} Pollicott, Mark {\em Rates of Convergence for Linear Actions of Cocompact Lattices on the Complex Plane }, 
Integers,   Volume 11B (2011), Proceedings of the Leiden Numeration Conference 2010. 

\bibitem[Ro1]{Roblin} Roblin, Thomas 
{\em Ergodicit\'e et \'equidistribution en courbure n\'egative} (French) [Ergodicity and uniform distribution in negative curvature], 
  Mem. Soc. Math. Fr. {\bf 95} (2003).


\bibitem[Ro2]{Roblin2} Roblin, Thomas {\em 
Sur l'ergodicit\'e rationnelle et les propri\'et\'es
ergodiques du flot g\'eod\'esique dans les vari\'et\'es
hyperboliques},
Ergod. Th. Dynam. Sys. (2000), {\bf 20}, 1785–1819
Printed in the United Kingdom
c

\bibitem[S]{Sarig} Sarig, Omri, {\em Invariant Radon measures for horocycle flows on Abelian covers}. 
Invent. Math. {\bf 157}, 519-551 (2004). 

\bibitem[Sa]{Sar} Sarnak, Peter
{\em Asymptotic behavior of periodic orbits of the horocycle flow and Einsenstein series},
Communications on Pure and Applied Mathematics {\bf 34} (1981), 719-739.

\bibitem[Sch0]{Scha0} Schapira Barbara, {\em Propri\'et\'es ergodiques du feuilletage horosph\'erique d'une 
vari\'et\'e \`a courbure n\'egative}, phd thesis,
{\tt \tiny http://tel.archives-ouvertes.fr/docs/00/16/34/20/PDF/schapira\_barbara\_these.pdf }

\bibitem[Sch1]{SchapETDS} Schapira Barbara, {\em On quasi-invariant transverse measures for the horospherical foliation of a negatively curved manifold.} Ergodic Theory Dynam. Systems {\bf 24} (2004), no. 1, 227–255.  
  
\bibitem[Sch2]{Scha1} Schapira Barbara, 
{\em Lemme de l'ombre et non divergence des horocycles d'une 
vari\'et\'e g\'eom\'etriquement finie},  Ann. Inst. Fourier (Grenoble) {\bf 54 } (2004),  no. 4, 939-987.

\bibitem[Sch3]{Scha2} Schapira, Barbara, 
{\em Equidistribution of the horocycles of a geometrically finite surface} 
Int. Math. Res. Not. {\bf 2005}, no. 40, 2447-2471.






\bibitem[S2]{Sullivan} Sullivan, Dennis
 {\em Entropy, Hausdorff measures old and new, and limit sets of 
geometrically finite Kleinian groups}, Acta Math. {\bf 153} (1984) no. 3-4, 259-277.

\end{thebibliography}
\end{document}